\documentclass[hidelinks,onefignum,onetabnum]{siamart220329}

\usepackage{amsmath} 
\usepackage{amssymb} 
\usepackage{amsfonts}
\usepackage[utf8]{inputenc}
\usepackage{graphicx}
\usepackage{float} 
\usepackage{subfigure} 
\usepackage{epstopdf}
\usepackage{multirow}
\usepackage{algorithmic}
\usepackage{tabularx} 
\ifpdf
  \DeclareGraphicsExtensions{.eps,.pdf,.png,.jpg}
\else
  \DeclareGraphicsExtensions{.eps}
\fi

\newsiamremark{remark}{Remark}
\newsiamremark{hypothesis}{Hypothesis}
\crefname{hypothesis}{Hypothesis}{Hypotheses}
\newsiamthm{claim}{Claim}

\headers{An IRJBD algorithm for large GSVD computations}{K. FANG and Z. JIA}

\title{An Implicitly Restarted Joint Bidiagonalization Algorithm for Large GSVD Computations\thanks{Submitted to the editors DATE.
\funding{This work was supported in part by the National Science Foundation of China (NSFC) under
grant 12171273.}}}

\author{Kaixiao Fang\thanks{Department of Mathematical Sciences, Tsinghua University, 100084 Beijing, People's Republic of China.
  (\email{fkx20@mails.tsinghua.edu.cn}).}
\and Zhongxiao Jia\thanks{The corresponding author.
Department of Mathematical Sciences, Tsinghua University, 100084 Beijing, People's Republic of China.
  (\email{jiazx@tsinghua.edu.cn}).}}

\usepackage{amsopn}

\begin{document}
\sloppy 

\maketitle

\begin{abstract}
    The joint bidiagonalization (JBD) process of a regular matrix pair $\{A,L\}$
    is mathematically equivalent to two simultaneous Lanczos bidiagonalization
    processes of
    the upper and lower parts of the Q-factor of QR factorization of the stacked matrix $(A^{\mathrm T},\,L^{\mathrm T})^{\mathrm T}$ when their starting
    vectors are closely related in a specific way. The resulting JBD method for computing extreme generalized singular values and corresponding left and right singular vectors of $\{A,L\}$ realizes the standard Rayleigh--Ritz projection
    of the generalized singular value decomposition (GSVD) problem of $\{A,L\}$ onto the
    two left and one right subspaces generated by the JBD process.  In this paper, the implicit
    restarting technique is nontrivially and skillfully extended to the JBD process, and an implicitly
    restarted JBD (IRJBD) algorithm is developed with proper selection of crucial shifts proposed and a few key implementation details addressed in finite precision arithmetic.
    Compact upper bounds are established for the residual norm of an approximate
    GSVD component in both exact and finite precision arithmetic, 
    which are used to design efficient and reliable stopping criteria and avoid the expensive computation of approximate right generalized singular vectors. 
    Numerical experiments illustrate that
    IRJBD performs well and is more efficient than the thick-restart JBD algorithm.
\end{abstract}



\section{Introduction}
Let $A \in \mathbb{R}^{m \times n}$ and $L \in \mathbb{R}^{p \times n}$ with $m+p\geq n$, and suppose the matrix pair $\{A, L\}$ is regular, i.e. $\operatorname{rank}((A^{\mathrm T},\,L^{\mathrm T})^{\mathrm T})=n$, where the superscript T denotes the transpose of a matrix or vector. Denote $q_1=\operatorname{dim}(\mathcal{N}(A))$, $q_2=\operatorname{dim}(\mathcal{N}(L))$, $q=n-q_1-q_2$, $l_1=\operatorname{dim}(\mathcal{N}(A^{\mathrm T}))$ and $l_2=\operatorname{dim}(\mathcal{N}(L^{\mathrm T}))$, where $\mathcal{N}(\cdot)$ is the null space of a matrix. Then according to \cite{Huang-Numerical-experiments} and \cite{PaigeandSaunders1981}, the generalized singular
value decomposition (GSVD) of $\{A, L\}$ is
\begin{equation}
\label{GSVD in matrix terms}
\left\{\begin{array}{l}
P_A^{\mathrm T} A X=C_A=\operatorname{diag}\left\{C, 0_{l_1, q_1}, I_{q_2}\right\}, \\
P_L^{\mathrm T} L X=S_L=\operatorname{diag}\left\{S, I_{q_1}, 0_{l_2, q_2}\right\},
\end{array}\right.
\end{equation}
where $P_A$ and $P_L$ are orthogonal, $X$ is nonsingular, $C=\operatorname{diag}(c_1, c_2, \ldots, c_q)$ and
$S=\operatorname{diag}(s_1, s_2, \ldots, s_q)$ with $0<c_i, s_i<1$, and $c_i^2+s_i^2=1$, and the subscripts
in the matrices $0$ and $I$ represent their row and columns numbers; see \cite{GolubandVan2013matrix} for more
details.

Partition $P_A,\ P_L$ and $X$ into
\begin{equation} \label{partition}
P_A=(P_{A,q},P_{A,l_1},P_{A,q_2}),\
P_L=(P_{L,q},P_{L,q_1},P_{L,l_2}),\
X=(X_q,X_{q_1},X_{q_2}),
\end{equation}
where the lowercase letters in subscripts
represent the column numbers of the matrices. Write $X_q=(x_1,\ldots,x_q)$, $P_{A,q}=(p_1^A, \ldots,
p_q^A)$ and $P_{L,q}=(p_1^L, \ldots, p_q^L)$. We call the
quintuples $(c_i,\,s_i,\,x_i,\,p_i^A,\,p_i^L),\ i=1,\ldots,q$, the
nontrival GSVD components of $\{A,L\}$ with the generalized singular values $c_i/s_i$ or
$\{c_i,s_i\}$, the left generalized
singular vectors $p_i^A$, $p_i^L$ and the right generalized
singular vectors $x_i$, and $(0_{l_1,q_1},I_{q_1},P_{A,l_1},P_{L,q_1},X_{q_1})$ and
$(I_{q_2},0_{l_2,q_2},P_{A,q_2},P_{L,l_2},X_{q_2})$ the trivial GSVD components (parts)
corresponding to the $q_1$ zero and $q_2$ infinite generalized singular values,
denoted by the pairs $\{c,s\}=\{0,1\}$ and  $\{c,s\}=\{1,0\}$, respectively.
The nontrivial part of GSVD in
\eqref{GSVD in matrix terms} can be written in the vector form
\begin{equation}
\label{GSVD in vector terms}
    \left\{\begin{array}{rl}
    A x_i & =c_i p_i^A, \\
    L x_i & =s_i p_i^L, \\
    s_i A^{\mathrm{T}} p_i^A & =c_i L^{\mathrm T} p_i^L. 
    \end{array} \quad i=1,\,2,\dots, q,\right.
\end{equation}

A number of methods have been available for the computation of several selected nontrivial
GSVD components. The Jacobi--Davidson type GSVD
method (JDGSVD) proposed by Hochstenbach \cite{Hochstenbach2009JDGSVD} computes several GSVD components
with the generalized singular values closest to a given target $\tau$. This method is suitable for
the case where $A$ or $L$ has
full column rank. It transforms the GSVD problem of $\{A,\,L\}$ into the generalized
eigenvalue problem of the augmented and cross-product matrix pair $\left\{\left[\begin{array}{cc}
0 & A \\
A^{\mathrm T} & 0
\end{array}\right],\,\left[\begin{array}{cc}
I & 0 \\
0 & L^{\mathrm T} L
\end{array}\right]\right\}$ when $L$ has full column rank (or $\left\{\left[\begin{array}{cc}
0 & L \\
L^{\mathrm T} & 0
\end{array}\right],\,\left[\begin{array}{cc}
I & 0 \\
0 & A^{\mathrm T} A
\end{array}\right]\right\}$ when $A$ has full column rank), and recovers GSVD components from the
computed generalized eigenpairs of the augmented and cross-product matrix pair.
Huang and Jia \cite{Huang-Numerical-experiments} have recently proposed a cross-product-free JDGSVD
method (CPF-JDGSVD) to compute several generalized singular values closest to a target point $\tau$ and the
corresponding generalized singular vectors. The CPF-JDGSVD method works on $A$ and $L$ directly and does not explicitly involve $A^{\mathrm T}A$
and $L^{\mathrm T}L$ in the extraction stage; it constructs the two left subspaces by multiplying the right
subspace by $A$ and $L$, respectively, and projects the GSVD problem of $\left\{A,\,L\right\}$ onto them.
The method is mathematically equivalent to implicitly implementing the standard Rayleigh--Ritz projection
of the
generalized eigenvalue problem of the matrix pair $\{A^{\mathrm T}A,\,L^{\mathrm
T}L\}$ onto the right subspace. In the subspace expansion stage, at each outer iteration one needs to iteratively solve an $n\times n$ correction equation, called inner iterations. The CPF-JDGSVD computes GSVD components more accurately than JDGSVD
\cite{Hochstenbach2009JDGSVD} does.  Huang and Jia \cite{HuangandJia2022harmonicJDGSVD} have also presented
two JDGSVD methods based on harmonic extraction approaches, which are more efficient for computing interior
GSVD components than the CPF-JDGSVD method; furthermore, they have developed
refined and refined harmonic JD algorithms \cite{HuangandJia2024RHJD}, which converge more smoothly and outperform the standard and harmonic JD
algorithms in \cite{Huang-Numerical-experiments,HuangandJia2022harmonicJDGSVD}, respectively, when computing
exterior and interior GSVD components. Jia and Zhang \cite{JiaandZhang2023FEASTGSVD} exploit the
Chebyshev--Jackson series expansion to propose a CJ-FEAST GSVDsolver for computing all the generalized
singular values in a given interval and the corresponding left and right generalized singular vectors.

In this paper we are concerned with the problem of computing several extreme
GSVD components corresponding to the nontrivial
largest or smallest generalized singular values. Such kind of problem
arises from many applications, e.g., general-form regularized solutions of linear discrete ill-posed problems
\cite{Rank-Deficient-and-Discrete-Ill-Posed-Problems}, principal component analysis in statistics \cite{PCA},
pattern recognition \cite{model-identification}, and signal processing \cite{Signal-Analysis}. The joint
bidiagonalization (JBD) method proposed by Zha \cite{Zha-JBD} can serve this purpose.
It makes use of the JBD
process to jointly bidiagonalize $\{A,\,L\}$ to upper bidiagonal forms, generates two explicit
left Krylov subspaces and one implicit right subspace that are closely related to certain Krylov
subspaces, and realizes the standard Rayleigh--Ritz projection \cite{Huang-Numerical-experiments,JiaandZhang2023FEASTGSVD} of the GSVD problem of $\{A,\,L\}$
onto the generated left and right subspaces to obtain approximations of extreme GSVD components.
Kilmer, Hansen and Español \cite{Kilmer-JBD}
adapt the JBD process to joint lower and upper bidiagonalizations of $\{A,\,L\}$,
and then exploit the process to solve linear discrete ill-posed problems with general-form Tikhonov regularization.
Jia and Yang \cite{JIAandYang2020JBDforTikhonov} make use of this JBD process to solve large-scale linear
discrete ill-posed problems with the general-form regularization, and propose an effective but
more efficient and robust
pure iterative solver without explicit Tikhonov regularization.

As we will see, the JBD process is mathematically equivalent to joint Lanczos lower and upper 
bidiagonalizations when their starting vectors are chosen in some special way, and the JBD method for
computing extreme, i.e., largest or smallest, GSVD components of $\{A,L\}$
amounts to the Lanczos bidiagonalization method for computing extreme SVD triplets of
certain two closely related single matrices resulting from $\{A,L\}$. As a result, the JBD method inherits
the convergence properties of Lanczos bidiagonalization method. For the computation of
extreme GSVD components, the subspace dimension needed may be large in order to ensure convergence.
For the large $\{A,L\}$, however, the subspace dimension must be
limited and cannot exceed some maximum allowed due to the storage and/or computational cost.
Therefore, for practical purpose, it is generally necessary to restart the JBD method for
a given maximum dimension by selecting
better initial vectors to generate increasingly better left and right subspaces. Proceed in such a way until a restarted JBD algorithm converges for a prescribed
stopping tolerance.

Sorensen \cite{implicit-restarted-Arnoldi} proposes an implicit restarting technique
for the Arnoldi method that solves large eigenvalue problems,
which can save computational cost considerably
compared to an explicitly restarted Arnoldi method at each restart for the same subspace dimension
and, meanwhile, maintains numerical stability. This technique is applicable or adaptable
to numerous Krylov subspace methodsd in a
variety of contexts, and has been widely used, e.g., the Lanczos method \cite{implicit-restart-Lanczos} and the block Lanczos method
\cite{implicit-restart-block-Lanczos} for symmetric eigenvalue problems, the refined Arnoldi method
\cite{Jia-poly-refined} and the refined harmonic Arnoldi method \cite{Jia-refined-harmonic-Arnoldi}
for large eigenvalue problems,
the Lanczos bidiagonalization method \cite{Jia-IRRBL,implicit-restart-bidiag} and the corresponding refined and harmonic
refined methods \cite{Jia-IRRBL,Jia-IRRHLB}  for
SVD problems. One of the keys for the success of an implicitly restarted
algorithm critically rely on proper selection of shifts involved. Each of these methods falls into one
of the standard extraction-, harmonic extraction-, refined extraction- and refined harmonic extraction-based Rayleigh--Ritz projections, and the best possible
shifts depend on the chosen extraction approach \cite{Jia-poly-refined,Jia-refined-harmonic-Arnoldi,Jia-IRRBL,implicit-restart-bidiag}.

For the JBD method, Alvarruiz, Campos and Roman \cite{TRJBD} have developed a
thick-restarted JBD (TRJBD) algorithm,
which will be later reviewed in this paper. In this paper, we extend the implicit restarting technique nontrivially and skillfully to the JBD method, and develop an implicitly restarted JBD (IRJBD) algorithm. We consider selection strategies of crucial shifts, and propose exact shifts that are optimal within the
framework of the JBD method
in some sense. We establish the relationship between the initial vector after and before implicit restart. Being directed against the disadvantage that the residual in \cite{Zha-JBD} is defined for the generalized eigenvalue problem of the matrix pair $\left\{A^{\mathrm T}A,\,L^{\mathrm T}L\right\}$ and misses the approximate left generalized singular vectors of $\left\{A,\,L\right\}$, we define a new and general residual that takes approximate left and right generalized singular vectors into consideration. We establish compact upper bounds for the residual norm in exact and finite precision arithmetic, respectively, and use them to design efficient and reliable stopping criteria, which avoid the expensive but unnecessary computation of approximate right generalized singular vectors before the convergence,
each of which requires accurate solution of a large least squares
problem with the coefficient matrix $(A^{\mathrm T},\,L^{\mathrm T})^{\mathrm T}$. Notably,
we prove the intrinsic property that the JBD method cannot compute zero and infinite trivial
GSVD components.

In section 2, we review the Lanczos lower and upper bidiagonalization processes. In section 3, we review the JBD process and the JBD method, and shed light on the connection between the JBD process and Lanczos bidiagonalizations. We present a general definition of residual of an approximate GSVD component, and derive compact bounds for the residual norm without involving approximate left and right generalized singular vectors
in exact and finite precision arithmetic, respectively. In section 4, we show how to implicitly restart the JBD process. 
Section 5 devotes to the selection of crucial shifts involved. Section 6 summarizes the IRJBD algorithm. Section 7 reports numerical experiments to illustrate the performance of IRJBD and its higher efficiency than TRJBD,
and section 8 concludes the paper.

Through the paper, 
denote by $\mathcal{K}_k(M,v)=\operatorname{span}(v,Mv,\dots,M^{k-1}v)$ the $k$-dimensional Krylov subspace
generated by a matrix $M$ and a unit length vector $v$, by $\epsilon$ the unit roundoff, and by
$e_i$ is the $i$-th column of the identity matrix $I_k$ of order $k$.

\section{Lanczos lower and upper bidiagonalizations}
Let the economy QR factorization of $\left\{A, L\right\}$ be
\begin{equation}
\label{QR}
    \left(\begin{array}{l}
    A \\
    L
    \end{array}\right)=Q R=\left(\begin{array}{l}
    Q_A \\
    Q_L
\end{array}\right) R
\end{equation}
with $Q_A \in \mathbb{R}^{m \times n}$ and $Q_L \in \mathbb{R}^{p \times n}$. Then
GSVD (\ref{GSVD in matrix terms}) is expressed by the CS decomposition (cf. \cite{GolubandVan2013matrix})
\begin{equation}
\label{CS of QA and QL}
Q_A=P_AC_AW^{\mathrm T},\quad Q_L=P_LS_LW^{\mathrm T}
\end{equation}
of the column orthonormal matrix $(Q_A^{\mathrm T},\,Q_L^{\mathrm T})^{\mathrm T}$ and
$$
X=R^{-1}W.
$$
Applying the $k$-step Lanczos lower and upper bidiagonalization processes \cite{Bjorck1996LeastSquares,LSQR}
to $Q_A$ and $Q_L$ with the starting vectors $u_1$ and $\hat v_1$, respectively, we obtain
\begin{equation}
\label{Lanczos bidiagonalization}
    \begin{aligned}
    & Q_A V_k=U_{k+1} B_k, & & Q_A^{\mathrm T} U_{k+1}=V_k B_k^{\mathrm T}+\alpha_{k+1} v_{k+1} e_{k+1}^{\mathrm T}, \\
    & Q_L \widehat{V}_k=\widehat{U}_k \widehat{B}_k, & & Q_L^{\mathrm T} \widehat{U}_k=\widehat{V}_k \widehat{B}_k^{\mathrm T}+\hat{\beta}_k \hat{v}_{k+1} e_k^{\mathrm T},
    \end{aligned}
\end{equation}
where
$$
\resizebox{\textwidth}{!}{$
B_k=\left(\begin{array}{cccc}
\alpha_1 & & & \\
\beta_2 & \alpha_2 & & \\
& \beta_3 & \ddots & \\
& & \ddots & \alpha_k \\
& & & \beta_{k+1}
\end{array}\right) \in \mathbb{R}^{(k+1) \times k}, \quad\widehat{B}_k=\left(\begin{array}{cccc}
\hat{\alpha}_1 & \hat{\beta}_1 & & \\
& \hat{\alpha}_2 & \ddots & \\
& & \ddots & \hat{\beta}_{k-1} \\
& & & \hat{\alpha}_k
\end{array}\right) \in \mathbb{R}^{k \times k},
$}
$$
$$
U_{k+1}=\left(u_1, \ldots, u_{k+1}\right) \in \mathbb{R}^{m \times(k+1)}, \quad V_k=\left(v_1, \ldots, v_k\right) \in \mathbb{R}^{n \times k},
$$
$$
\widehat{U}_k=\left(\hat{u}_1, \ldots, \hat{u}_k\right) \in \mathbb{R}^{p \times k}, \quad \widehat{V}_k=\left(\hat{v}_1, \ldots, \hat{v}_k\right) \in \mathbb{R}^{n \times k}
$$
with $U_{k+1},\,V_k,\,\widehat{U}_k$ and $\widehat{V}_k$ being column orthonormal, and $v_1=\frac{Q_A^{\mathrm{T}}u_1}{\left\|Q_A^{\mathrm{T}}u_1\right\|}$.

If $\hat{v}_1:=v_1$, 
then it is proved in \cite{Kilmer-JBD,Zha-JBD} that
\begin{equation}
\label{relation of v and vhat}
\hat{v}_{i+1}=(-1)^i v_{i+1}, \quad \hat{\alpha}_i \hat{\beta}_i=\alpha_{i+1} \beta_{i+1}, \quad i=1,2,\dots,k-1.
\end{equation}
From (\ref{Lanczos bidiagonalization}) we have
\begin{equation}
\label{relation to Lanczos}
\resizebox{0.9\textwidth}{!}{$
\begin{aligned}
&Q_A^{\mathrm T} Q_A V_k=V_k B_k^{\mathrm T} B_k+\alpha_{k+1} \beta_{k+1} v_{k+1} e_k^{\mathrm T},&&Q_A Q_A^{\mathrm T} U_{k+1} =U_{k+1} B_k B_k^{\mathrm T}+\alpha_{k+1} Q_A v_{k+1} e_{k+1}^{\mathrm T}, \\
&Q_L^{\mathrm T} Q_L \widehat V_k=\widehat{V}_k \widehat{B}_k^{\mathrm T} \widehat{B}_k+\hat\alpha_k \hat\beta_k \hat{v}_{k+1} e_k^{\mathrm T},&&Q_L Q_L^{\mathrm T} \widehat{U}_k =\widehat{U}_k \widehat{B}_k \widehat{B}_k^{\mathrm T}+\hat{\beta}_k Q_L \hat{v}_{k+1}e_k^{\mathrm T}. \\
\end{aligned}
$}
\end{equation}
The first and third relations in (\ref{relation to Lanczos}) correspond to the symmetric Lanczos processes of $Q_A^{\mathrm{T}}Q_A$ and $Q_L^{\mathrm{T}}Q_L$ with the starting vectors $v_1$ and $\hat v_1$, respectively \cite{Bjorck1996LeastSquares,Parlett1980symmetric}. The second and fourth ones will be used to establish the expressions of the updated starting vectors of the Lanczos lower and upper bidiagonalization processes after implicit restart.

Notice from \eqref{relation of v and vhat} and \eqref{relation to Lanczos} that the subspaces $\operatorname{span}(V_k)=\operatorname{span}(\widehat{V}_k)=\mathcal{K}_k(Q_A^{\rm T}Q_A,v_1)=\mathcal{K}_k(Q_L^{\rm T}Q_L,v_1)$, ${\rm span}(U_{k+1})=\mathcal{K}_{k+1}(Q_AQ_A^{\rm T},u_1)$
and  ${\rm span}(\widehat{U}_k)=\mathcal{K}_k(Q_LQ_L^{\rm T},\hat{u}_1)$ .
It is known from (\ref{Lanczos bidiagonalization}) that $B_k=U_{k+1}^{\mathrm T}Q_AV_k$ is the Ritz--Galerkin projection matrix of $Q_A$ on the left subspace $\operatorname{span}(U_{k+1})$ and the right subspace $\operatorname{span}(V_k)$, and $\widehat B_k=\widehat{U}_k^{\mathrm T}Q_L\widehat{V}_k$ is the Ritz--Galerkin projection matrix of $Q_L$ on the left subspace $\operatorname{span}(\widehat U_k)$ and the right subspace $\operatorname{span}(V_k)$. Therefore, the extreme singular values of $Q_A$ and $Q_L$ can be approximated by those of $B_k$ and $\widehat B_k$ \cite{Bjorck1996LeastSquares,Parlett1980symmetric}.

\section{The JBD process and the JBD method}
\subsection{The JBD process}
For $A$ and $L$ large, the QR factorization (\ref{QR}) is generally unaffordable due to
its excessive storage and expensive computational cost.
Thus, we suppose $Q$ and $R$ are not available in practical computation. For $\hat v_1=v_1$, the Lanczos bidiagonalization processes in (\ref{Lanczos bidiagonalization}) can be realized by the JBD process \cite{Jia-JBD-finite-precision,Kilmer-JBD}, described in Algorithm \ref{alg:JBD}.

\begin{algorithm}[H]
	\renewcommand{\algorithmicrequire}{\textbf{Input:}}
	\renewcommand{\algorithmicensure}{\textbf{Output:}}
	\caption{The $k$-step JBD process of $\left\{A,L\right\}$}
	\label{alg:JBD}
	\begin{algorithmic}[1]
	    \REQUIRE $u_1 \in \mathbb{R}^m$ with $\|u_1\|=1$
        \STATE $\alpha_1 v_1^{\prime}=QQ^{\mathrm{T}}\left(\begin{array}{c}
        u_{1} \\
        0_p
        \end{array}\right)$
        \STATE $\hat{\alpha}_1 \hat{u}_1=v_1^{\prime}(m+1: m+p)$
        \FOR{$i=1,2,...,k$}
		\STATE $\beta_{i+1} u_{i+1}=v_i^{\prime}(1: m)-\alpha_i u_i$
            \label{JBD inner start}
        \STATE $\alpha_{i+1} v_{i+1}^{\prime}=QQ^{\mathrm{T}}\left(\begin{array}{c}
        u_{i+1} \\
        0_p
        \end{array}\right)-\beta_{i+1}v_i^{\prime}$
        \STATE $\hat{\beta}_i=\left(\alpha_{i+1} \beta_{i+1}\right)/\hat{\alpha}_i$
        \STATE $\hat{\alpha}_{i+1} \hat{u}_{i+1}=(-1)^i v_{i+1}^{\prime}(m+1: m+p)-\hat{\beta}_i \hat{u}_i$
        \label{JBD inner end}
	    \ENDFOR
	\end{algorithmic}
\end{algorithm}

Notice that $QQ^{\mathrm{T}}(u_i^{\mathrm{T}},\,0_p^{\mathrm{T}})^{\mathrm{T}}$ in steps 1 and 5 are the orthogonal projections of $(u_i^{\mathrm{T}},\,0_p^{\mathrm{T}})^{\mathrm{T}}$, $i=1,2,\dots,k+1$ onto the column space of $(A^{\mathrm{T}},\,L^{\mathrm{T}})^{\mathrm{T}}$. Therefore,
$$
QQ^{\mathrm{T}}(u_i^{\mathrm{T}},\,0_p^{\mathrm{T}})^{\mathrm{T}}=(A^{\mathrm{T}},\,L^{\mathrm{T}})^{\mathrm{T}}\tilde x_i,
$$
where the vectors $\tilde x_i$ solve the large least squares problems:
\begin{equation}
\label{least-squares}
\tilde x_i=\arg \min \limits_{x \in \mathbb{R}^n}\left\|\left(\begin{array}{l}A \\ L\end{array}\right) x-\left(\begin{array}{l}u_i \\ 0_p\end{array}\right)\right\|,\quad i=1,\dots,k+1,
\end{equation}
which are supposed to be computed iteratively by the commonly used LSQR algorithm \cite{LSQR}. In such a way, we avoid the QR factorization (\ref{QR}) in Algorithm \ref{alg:JBD}.

Suppose the least squares problems in Algorithm \ref{alg:JBD} are solved accurately. The JBD process can be written in the matrix form
\begin{equation}
\label{JBD}
    \begin{aligned}
    & \left(I_m, 0_{m,p}\right) V_k^{\prime}=U_{k+1} B_k, \\
    & Q Q^{\mathrm T}\left(\begin{array}{c}
    U_{k+1} \\
    0_{p,k+1}
    \end{array}\right)=V_k^{\prime} B_k^{\mathrm T}+\alpha_{k+1} v_{k+1}^{\prime} e_{k+1}^{\mathrm T}, \\
    & \left(0_{p,m}, I_p\right) V_k^{\prime} D_k=\widehat{U}_k \widehat{B}_k,
    \end{aligned}
\end{equation}
where $V_k^{\prime}=QV_k$ and $D_k=\operatorname{diag}\left(1,-1, \ldots,(-1)^{k-1}\right) \in \mathbb{R}^{k \times k}$.
Let 
\begin{equation}\label{bkbar}
H_k=R^{-1} V_k=\left(h_1, \ldots, h_k\right),\ \ \bar{B}_k=\widehat{B}_k D_k.
\end{equation}
Then
\begin{equation}
\label{BTB+barBTbarB=I}
\begin{aligned}
&A H_k=U_{k+1} B_k,\quad LH_k=\widehat{U}_k \bar{B}_k,\\
&B_k^{\mathrm T} B_k+\bar{B}_k^{\mathrm T} \bar{B}_k=I_k.
\end{aligned}
\end{equation}

\subsection{The JBD method}\label{sec:JBD method}

The JBD process \eqref{AH=UB} generates two explicit left Krylov subspaces
${\rm span}(U_{k+1}),\ {\rm span}(\widehat{U}_k)$
and one implicit right Krylov subspace ${\rm span}(H_k)=R^{-1}{\rm span}(V_k)$ as
$H_k=R^{-1}V_k$ with $R$ unknown in computation,
which are exploited to compute some approximate GSVD
components of $\{A,\,L\}$.
Let the singular values of $B_k$ and $\bar B_k$ be $\{c_i^{(k)}\}_{i=1}^k$ and $\{s_i^{(k)}\}_{i=1}^k$, respectively, with $\{c_i^{(k)}\}_{i=1}^k$ labeled in decreasing order and $\{s_i^{(k)}\}_{i=1}^k$ in increasing order. Identity (\ref{BTB+barBTbarB=I}) shows
$$
(c_i^{(k)})^2+(s_i^{(k)})^2=1,\quad i=1,\dots,k.
$$
Suppose Algorithm \ref{alg:JBD} does not break down before step $k$. Then the subdiagonal elements $\left\{\alpha_i\beta_i\right\}_{i=2}^{k}$ of $B_{k}^{\mathrm T}B_{k}$ are not equal to 0, meaning that $B_{k}^{\mathrm T}B_{k}$ is an unreduced symmetric tridiagonal matrix. By the eigenvalue strict interlacing theorem \cite[pp. 203]{Parlett1980symmetric}, the eigenvalues of $B_{k-1}^{\mathrm T}B_{k-1}$ and $B_{k}^{\mathrm T}B_{k}$ satisfy
$$
(c_1^{(k)})^2>(c_1^{(k-1)})^2>(c_2^{(k)})^2>\cdots>(c_{k-1}^{(k)})^2>(c_{k-1}^{(k-1)})^2>(c_{k}^{(k)})^2.
$$
Thus, the singular values $\{c_i^{(k)}\}_{i=1}^k$ of $B_k$ are distinct, and so are $\{s_i^{(k)}\}_{i=1}^k$. Moreover, since $B_k$ and $\bar B_k$ have full column ranks, we have $0<c_i^{(k)},\,s_i^{(k)}<1$ for $i=1,2,\dots,k$.

Let the SVDs of $B_k$ and $\bar B_k$ or, equivalently, the GSVD of $\{B_k,\,\bar B_k\}$ be
\begin{equation}
\label{GSVD of Bk and Bkbar}
\begin{aligned}
    B_k=P_k \Theta_k W_k^{\mathrm T}, \quad \Theta_k=\operatorname{diag}\left(c_1^{(k)}, \ldots, c_k^{(k)}\right), \quad 1> c_1^{(k)}>\cdots> c_k^{(k)}> 0,\\
    \bar{B}_k=\bar{P}_k \Psi_k W_k^{\mathrm T}, \quad \Psi_k=\operatorname{diag}\left({s}_1^{(k)}, \ldots, {s}_k^{(k)}\right), \quad 0< s_1^{(k)}<\cdots<s_k^{(k)}< 1,
\end{aligned}
\end{equation}
where $P_k$ is orthonormal, $\bar P_k$ and $W_k$ are orthogonal. Then the JBD method computes
$k$ approximate generalized singular values $\{c_i^{(k)},s_i^{(k)}\}$, called the Ritz values,
and right generalized singular vectors $x_i^{(k)}=H_k w_i^{(k)}=R^{-1} V_k w_i^{(k)}$ and left generalized singular vectors $y_i^{(k)}=U_{k+1} p_i^{(k)}$ and $z_i^{(k)}=\widehat{U}_k \bar{p}_i^{(k)}$ associated with $A$ and $L$, called the right and left Ritz vectors, respectively.
It follows from the convergence theory of Lanczos bidiagonaliztion method on
$Q_A$ and $Q_L$ \cite{Bjorck1996LeastSquares} and is further known from \cite{Jia-JBD-finite-precision,JIAandYang2020JBDforTikhonov} that the JBD method generally favors extreme GSVD components of $\{A,L\}$.

Notice from (\ref{JBD}) that
\begin{equation}
\label{calculate x_i by lsqr}
\left(\begin{array}{c}
A \\
L
\end{array}\right) x_i^{(k)}=Q R R^{-1} V_k w_i^{(k)}=V^{\prime}_k w_i^{(k)}.
\end{equation}
Therefore, we obtain $x_i^{(k)}$ by solving this consistent rectangular
linear system using the LSQR algorithm.
In the analysis, we suppose that this system is solved accurately.

Remarkably, if $\{A,\,L\}$ has trivial zero or infinite GSVD components, then one of the approximate left generalized singular vectors never converges, as shown below.

\begin{theorem}
\label{Th:left singular vector corresponding to zero and infinity}
Suppose \Cref{alg:JBD} does not break down before step $k$, and let
$P_{A,l_1}$ and $P_{L,l_2}$ be defined as in \eqref{partition}.
\begin{enumerate}
\item If $\{A,\,L\}$ has a zero generalized singular value and the unit-length $p_0\in \operatorname{span}(P_{A,l_1})$ is
any left generalized singular vector associated with $A$ corresponding to the zero generalized singular
value, then
\begin{equation}\label{distzero}
\sin\angle(p_0,\operatorname{span}(U_{k+1}))=\sin\angle(p_0,\pm u_1).
\end{equation}

\item If $\{A,\,L\}$ has an infinite generalized singular value and the unit-length $p_{\infty}\in \operatorname{span}(P_{L,l_2})$
is any left generalized singular vector associated with $L$ corresponding to the infinite generalized singular value, then
\begin{equation}\label{distinf}
\sin\angle(p_{\infty},\operatorname{span}(\widehat U_{k}))=\sin\angle(p_{\infty},\pm \hat u_1).
\end{equation}
\end{enumerate}
\end{theorem}

\begin{proof}
From (\ref{relation to Lanczos}), we know that $\operatorname{span}(U_{k+1})=\mathcal{K}_{k+1}(Q_AQ_A^{\mathrm T},u_1)$ and $\operatorname{span}(\widehat U_k)=\mathcal{K}_k(Q_LQ_L^{\mathrm T},\hat u_1)$. For item 1,
from \eqref{GSVD in matrix terms} and \eqref{CS of QA and QL}
we have $Q_A^{\mathrm T}p_0=0$. Make the orthogonal direct sum decomposition
$u_1=\gamma p_0+p_0^{\perp}$, where $\gamma=u_1^{\mathrm T}p_0$ with
 $|\gamma|\leq 1$ and $p_0^{\mathrm T}p_0^{\perp}=0$. Then $(Q_AQ_A^{\mathrm T})^iu_1=(Q_AQ_A^{\mathrm T})^ip_0^{\perp}$
for any integer $i\geq 1$. For any unit length vector $u\in\operatorname{span}(U_{k+1})$, we can write
$$
u=\gamma_0u_1+\gamma_1Q_AQ_A^{\mathrm T}u_1+\cdots+\gamma_k(Q_AQ_A^{\mathrm T})^ku_1.
$$
Then it follows from the above that
$u^{\mathrm T}p_0=\gamma_0\gamma$. Therefore,
$$
\sin\angle(p_0,u)=\sqrt{1-(u^{\mathrm T}p_0)^2}\geqslant\sqrt{1-\gamma^2}
=\sin\angle(p_0,u_1),
$$
where the equality holds if and only if $|\gamma_0|=1$, i.e.,  $u=\pm u_1$.
We thus obtain
$$
\sin\angle(p_0,\operatorname{span}(U_{k+1}))=\min_{u\in \operatorname{span}(U_{k+1})}
\sin\angle(p_0,u)=\sin\angle(p_0,\pm u_1).
$$
This proves \eqref{distzero}. The proof of \eqref{distinf} is similar.
\end{proof}

Theorem \ref{Th:left singular vector corresponding to zero and infinity} indicates that,
unless $u_1=\pm p_0$ or $\hat u_1=\pm p_{\infty}$,  the quantities $\sin\angle(p_0,\operatorname{span}(U_{k+1}))$ and $\sin\angle(p_{\infty},\operatorname{span}(\widehat U_{k}))$ are {\em positive} constants and thus {\em never} tend to zero as $k$ increases. For the JBD method,
since the approximate left generalized singular vectors $y_i^{(k)}\in
\operatorname{span}(U_{k+1})$ and $z_i^{(k)}\in \operatorname{span}(\widehat{U}_k)$,
it is impossible for them to converge to $p_0$ or $p_{\infty}$. Therefore, the JBD method
cannot compute the zero or infinite GSVD components.


\subsection{New residual, compact bounds for its norm, and a reliable and efficient stopping criterion}
Now we concentrate on designing a general-purpose and reliable stopping criterion for the JBD method.
Since the generalized singular values and right generalized singular vectors mathematically
are the solutions to the generalized eigenvalue problem $s_i^2 A^{\mathrm T} A x_i=c_i^2 L^{\mathrm T} L x_i$, Zha in \cite{Zha-JBD} defines the residual 
\begin{equation}
\label{old residual}
r_{i,old}^{(k)}=((s_i^{(k)})^2 A^T A-(c_i^{(k)})^2 L^T L) x_i^{(k)},
\end{equation}
which is precisely the residual of an approximate eigenpair of the generalized eigenvalue problem of $\{A^{\mathrm T}A, L^{\mathrm T}L\}$. Zha derives a compact bound for its norm:
\begin{equation}
\label{old residual bound}
\|r_{i,old}^{(k)}\|\leqslant\|R\| \alpha_{k+1} \beta_{k+1}|e_k^T w_i^{(k)}|,
\end{equation}
and uses this bound to design a stopping criterion of the JBD method without explicitly
computing $x_i^{(k)}$ before the occurrence of its convergence.

The residual (\ref{old residual}) has a severe limitation because it does not take into account the approximate left generalized singular vectors $y_i^{(k)}$ for $A$ and $z_i^{(k)}$ for $L$. Since the convergence behavior of approximate left and right generalized singular vectors may differ greatly, it is very likely that approximate left generalized singular vectors may have much poorer accuracy than the corresponding right one, as \Cref{Th:left singular vector corresponding to zero and infinity} indicates.
That residual (\ref{old residual}) tends to zero only means the convergence of the approximate right generalized singular vector but not that of the two approximate left ones. Therefore,
using the size of residual (\ref{old residual}) to judge the convergence of the JBD method
is not reliable and can be misleading. In terms of GSVD (\ref{GSVD in vector terms}), we define 
a new general residual that takes $y_i^{(k)}$ and $z_i^{(k)}$ into consideration, as Definition \ref{Def:residual} shows, and use it to
design a general-purpose stopping criterion for the JBD method.

\begin{definition}
\label{Def:residual}
The residual of an approximate GSVD component $(c_i^{(k)}$, $s_i^{(k)}$, $x_i^{(k)}$, $y_i^{(k)}$, $z_i^{(k)})$ is defined by
\begin{equation}
\label{residual}
r_i^{(k)}=\left(\begin{array}{c}
r_{i,1}^{(k)} \\
r_{i,2}^{(k)} \\
r_{i,3}^{(k)}
\end{array}\right):=\left(\begin{array}{c}
A x_i^{(k)}-c_i^{(k)} y_i^{(k)} \\
L x_i^{(k)}-s_i^{(k)} z_i^{(k)} \\
s_i^{(k)} A^{\mathrm{T}} y_i^{(k)}-c_i^{(k)} L^{\mathrm{T}} z_i^{(k)}
\end{array}\right).
\end{equation}
\end{definition}

It is costly to compute residual (3.8) since one must compute the approximate right generalized singular vector $x_i^{(k)}$ explicitly by solving the large linear system (\ref{calculate x_i by lsqr}). To design an efficient and reliable stopping criterion, we first establish compact bounds for the norm of residual (\ref{residual}) in exact arithmetic.

\begin{theorem}
\label{Th:bound of residual}
It holds that
\begin{equation}
\label{bound of residual}
\|r_i^{(k)}\|\leqslant\|R\|\left|s_i^{(k)} \alpha_{k+1} e_{k+1}^{\mathrm{T}} p_i^{(k)}-c_i^{(k)} \bar{\beta}_k e_k^{\mathrm{T}} \bar{p}_i^{(k)}\right|,\quad i=1,\dots,k
\end{equation}
and
\begin{equation}
\label{bound 2 of residual}
\|r_i^{(k)}\|\leqslant\|R\|\left|\frac{\alpha_{k+1} \beta_{k+1}}{c_i^{(k)} s_i^{(k)}} e_k^{\mathrm{T}} w_i^{(k)}\right|,\quad i=1,\dots,k.
\end{equation}
\end{theorem}

\begin{proof}
According to (\ref{BTB+barBTbarB=I}), we have
$$
\begin{aligned}
r_{i, 1}^{(k)} & =A H_k w_i^{(k)}-c_i^{(k)} U_{k+1} p_i^{(k)} \\
& =U_{k+1} B_k w_i^{(k)}-c_i^{(k)} U_{k+1} p_i^{(k)} \\
& =0, \\
r_{i, 2}^{(k)} & =L H_k w_i^{(k)}-s_i^{(k)} \widehat{U}_k \bar{p}_i^{(k)} \\
& =\widehat{U}_k \bar{B}_k w_i^{(k)}-s_i^{(k)} \widehat{U}_k \bar{p}_i^{(k)} \\
& =0.
\end{aligned}
$$
From (\ref{Lanczos bidiagonalization}), $h_{k+1}=R^{-1}v_{k+1}$ and $H_k=R^{-1} V_k$, we obtain
$$
\begin{aligned}
A^{\mathrm T}U_{k+1} & =R^{\mathrm T}Q_A^{\mathrm T}U_{k+1} \\
& =R^{\mathrm T}\left(V_k B_k^{\mathrm T}+\alpha_{k+1} v_{k+1} e_{k+1}^{\mathrm T}\right) \\
& =R^{\mathrm{T}}R\left(H_k B_k^{\mathrm{T}}+\alpha_{k+1} h_{k+1} e_{k+1}^{\mathrm{T}}\right). \\
L^{\mathrm T}\widehat U_{k} & =R^{\mathrm T}Q_L^{\mathrm T}\widehat U_{k} \\
& =R^{\mathrm T}\left(\widehat{V}_k \widehat{B}_k^{\mathrm T}+\hat{\beta}_k \hat{v}_{k+1} e_k^{\mathrm T}\right) \\
& =R^{\mathrm{T}}R\left(H_k \bar B_k^{\mathrm{T}}+\bar\beta_{k} h_{k+1} e_{k}^{\mathrm{T}}\right).
\end{aligned}
$$
Therefore,
\begin{eqnarray*}
r_{i, 3}^{(k)} & =&s_i^{(k)} A^{\mathrm T} U_{k+1} p_i^{(k)}-c_i^{(k)} L^{\mathrm{T}} \widehat{U}_k \bar{p}_i^{(k)}\\
& =& R^{\mathrm{T}}R\left[s_i^{(k)}\left(H_k B_k^{\mathrm{T}}+\alpha_{k+1} h_{k+1} e_{k+1}^{\mathrm{T}}\right) p_i^{(k)}-c_i^{(k)}\left(H_k \bar{B}_k^{\mathrm{T}}+\bar{\beta}_k h_{k+1} e_k^{\mathrm{T}}\right) \bar{p}_i^{(k)}\right] \\
& =&R^{\mathrm{T}}R\left(s_i^{(k)} c_i^{(k)} H_k w_i^{(k)}+s_i^{(k)} \alpha_{k+1} h_{k+1} e_{k+1}^{\mathrm{T}} p_i^{(k)}-c_i^{(k)} s_i^{(k)} H_k w_i^{(k)}-c_i^{(k)} \bar{\beta}_k h_{k+1} e_k^{\mathrm{T}} \bar{p}_i^{(k)}\right) \\
& =& R^{\mathrm{T}}\left(s_i^{(k)} \alpha_{k+1} e_{k+1}^{\mathrm{T}} p_i^{(k)}-c_i^{(k)} \bar{\beta}_k e_k^{\mathrm{T}} \bar{p}_i^{(k)}\right)v_{k+1}.
\end{eqnarray*}
As a result, (\ref{bound of residual}) holds because $\|r_i^{(k)}\|=\sqrt{\|r_{i,1}^{(k)}\|^2+\|r_{i,2}^{(k)}\|^2+\|r_{i,3}^{(k)}\|^2}$.

From (\ref{GSVD of Bk and Bkbar}), we have
$$
B_kw_i^{(k)}=c_i^{(k)}p_i^{(k)},\quad \bar B_kw_i^{(k)}=s_i^{(k)}\bar p_i^{(k)},\quad i=1,\dots,k.
$$
Substituting them into the last relation yields
\begin{eqnarray*}
r_{i, 3}^{(k)} & =&R^{\mathrm{T}}\left(\frac{s_i^{(k)}}{c_i^{(k)}} \alpha_{k+1} e_{k+1}^{\mathrm{T}} B_k w_i^{(k)}-\frac{c_i^{(k)}}{s_i^{(k)}} \bar{\beta}_k e_k^{\mathrm{T}}\bar B_k w_i\right)v_{k+1} \\
&=&R^{\mathrm{T}}\left(\frac{s_i^{(k)}}{c_i^{(k)}} \alpha_{k+1} \beta_{k+1} e_k^{\mathrm{T}} w_i^{(k)}+\frac{c_i^{(k)}}{s_i^{(k)}} \hat{\alpha}_k \hat{\beta}_k e_k^{\mathrm{T}} w_i\right)v_{k+1} \\
& =&R^{\mathrm{T}} \frac{\alpha_{k+1} \beta_{k+1}}{c_i^{(k)} s_i^{(k)}} e_k^{\mathrm{T}} w_i^{(k)} v_{k+1}. 
\hspace{2cm} \Box
\end{eqnarray*}
\end{proof}

\begin{remark}
We can similarly prove
$$
r_{i,old}^{(k)}=R^{\mathrm{T}} \alpha_{k+1} \beta_{k+1} e_k^{\mathrm{T}} w_i^{(k)} v_{k+1}.
$$
Therefore, $\|r_i^{(k)}\|>\|r_{i,old}^{(k)}\|$ unconditionally as
$c_i^{(k)} s_i^{(k)}<1$. Furthermore, if the desired generalized singular values are large or very small, then the quantity $c_i^{(k)} s_i^{(k)}$ will become very small as the JBD method converges, so that $\|r_i^{(k)}\|\gg\|r_{i,old}^{(k)}\|$. On the other hand,
this theorem indicates that the JBD method cannot compute zero or infinite GSVD components because $c_i^{(k)} s_i^{(k)}$ tend to $0$ as $k$ increases, which is consistent with Theorem \ref{Th:left singular vector corresponding to zero and infinity} and will also be illustrated in our later experiments.
\end{remark}

Since $R$ is not available, we need an estimate of $\|R\|$ in \eqref{bound of residual}
and \eqref{bound 2 of residual} in order to design a stopping criterion. We suggest to use the following
replacement:
\begin{equation}
\label{bound of R}
\|R\|\leqslant\sqrt{\|A\|_1\|A\|_{\infty}+\|L\|_1\|L\|_{\infty}},
\end{equation}
which can be calculated cheaply.
As a result, we can estimate the residual norm by its bound in (\ref{bound of residual}) or (\ref{bound 2 of residual}) without explicitly forming the approximate left and right generalized singular vectors. This way establishes a very efficient stopping criterion.

We need to compute both $p_i^{(k)}$ and $\bar p_i^{(k)}$ in the bound of (\ref{bound of residual}). Thus, the SVDs of $B_k$ and $\bar B_k$ or equivalently their GSVD are required. Bounds (\ref{bound 2 of residual}) and (\ref{bound of residual}) are mathematically the same. But if we only want to compute the SVD of $B_k$ or $\bar B_k$, we can use the bound in (\ref{bound 2 of residual}) instead and replace $s_i^{(k)}$ by $\sqrt{1-(c_i^{(k)})^2}$ or replace $c_i^{(k)}$ by $\sqrt{1-(s_i^{(k)})^2}$. In our implementation, we use (\ref{bound of residual}).

In finite precision arithmetic, (\ref{bound of residual}) and (\ref{bound 2 of residual}) do not hold due to the rounding errors. In what follows we establish two compact
bounds for the residual norm
to reveal how rounding errors to affect stopping tolerances and criteria. Without confusion, we use the same notations as before to denote the computed quantities.

We will make use of the following results in \cite{Jia-JBD-finite-precision} to 
derive bounds for $\|r_i^{(k)}\|$.

\begin{theorem}
\label{Th:JBD in finite}
(Theorems 3.2 and 3.4 in \cite{Jia-JBD-finite-precision}) Suppose the least squares problems in {\rm (\ref{least-squares})} are solved accurately. Then in finite precision arithmetic,
\begin{equation}
\label{JBD in finite}
    \begin{aligned}
    & Q_A V_k=U_{k+1} B_k+F_k, & & Q_A^{\mathrm T} U_{k+1}=V_k B_k^{\mathrm T}+\alpha_{k+1} v_{k+1} e_{k+1}^{\mathrm T}+G_{k+1}, \\
    & Q_L \widehat{V}_k=\widehat{U}_k \widehat{B}_k+\widehat{F}_k, & & Q_L^{\mathrm T} \widehat{U}_k=\widehat{V}_k \widehat{B}_k^{\mathrm T}+\hat{\beta}_k \hat{v}_{k+1} e_k^{\mathrm T}+\widehat{G}_k,
    \end{aligned}
\end{equation}
where $\|F_k\|=O(\|\underline{B}_k^{-1}\| \epsilon),\,\|G_{k+1}\|=O(\epsilon),\,\|\widehat{F}_k\|=O
(\|\underline{B}_k^{-1}\| \epsilon)$ and
$\|\widehat{G}_{k}\|=O(\|\underline{B}_k^{-1}\|\|\widehat{B}_k^{-1}\|\epsilon)$, with $\epsilon$ being the machine precision and $\underline{B}_k$ being the $k$-by-$k$ leading principal submatrix of $B_k$.
\end{theorem}

Theorem \ref{Th:JBD in finite} indicates that the relations in (\ref{Lanczos bidiagonalization}) hold within the level of $\epsilon$ in finite precision if and only if $\|\underline{B}_k^{-1}\|$ and $\|\widehat{B}_k^{-1}\|$ are modestly sized.
Based on Theorem \ref{Th:JBD in finite}, we can establish the following upper bounds for the residual norm in finite precision.

\begin{theorem}
\label{Th:bound of residual in finite}
Suppose the inner least squares problems in {\rm(\ref{least-squares})} are solved accurately. Then in finite precision,
\begin{equation}\label{boundsfinite}
\begin{aligned}
&\|r_i^{(k)}\|\leqslant\left\|R\right\|\left(\left|s_i^{(k)} \alpha_{k+1} e_{k+1}^{\mathrm{T}} p_i^{(k)}-c_i^{(k)} \bar{\beta}_k e_k^{\mathrm{T}} \bar{p}_i^{(k)}\right|+O\left(\left\|\underline{B}_k^{-1}\right\|\|\widehat{B}_k^{-1}\|\epsilon\right)\right), \\
&\|r_i^{(k)}\|\leqslant\left\|R\right\|\left(\left|\frac{\alpha_{k+1} \beta_{k+1}}{c_i^{(k)} s_i^{(k)}} e_k^{\mathrm{T}} w_i^{(k)}\right|+O\left(\left\|\underline{B}_k^{-1}\right\|\|\widehat{B}_k^{-1}\|\epsilon\right)\right).
\end{aligned}
\end{equation}
\end{theorem}

\begin{proof}
Substituting (\ref{JBD in finite}) into the proof of Theorem \ref{Th:bound of residual}, we obtain
$$
\resizebox{\textwidth}{!}{$
\begin{aligned}
r_{i, 1}^{(k)}& =(U_{k+1} B_k+F_k) w_i^{(k)}-c_i^{(k)} U_{k+1} p_i^{(k)}=F_k w_i^{(k)}, \\
r_{i, 2}^{(k)} & =(\widehat{U}_k \bar{B}_k+\widehat F_kD_k) w_i^{(k)}-s_i^{(k)} \widehat{U}_k \bar{p}_i^{(k)}=\widehat F_kD_k w_i^{(k)}, \\
r_{i, 3}^{(k)} & =R^{\mathrm{T}}\left[s_i^{(k)}\left(V_k B_k^{\mathrm{T}}+\alpha_{k+1} v_{k+1} e_{k+1}^{\mathrm{T}}+G_{k+1}\right) p_i^{(k)}-c_i^{(k)}\left(\widehat{V}_k \widehat{B}_k^{\mathrm{T}}+\hat{\beta}_k \hat{v}_{k+1} e_k^{\mathrm{T}}+\widehat G_k\right) \bar{p}_i^{(k)}\right] \\
& =R^{\mathrm{T}}\left[\left(s_i^{(k)} \alpha_{k+1} e_{k+1}^{\mathrm{T}} p_i^{(k)}-c_i^{(k)} \bar{\beta}_k e_k^{\mathrm{T}} \bar{p}_i^{(k)}\right)v_{k+1}+s_i^{(k)}G_{k+1}p_i^{(k)}-c_i^{(k)}\widehat G_{k}\bar p_i^{(k)}\right].
\end{aligned}
$}
$$
Therefore,
$$
\begin{aligned}
\|r_i^{(k)}\|^2&\leqslant O\left(\left\|\underline{B}_k^{-1}\right\| \epsilon\right)^2+O\left(\left\|\underline{B}_k^{-1}\right\| \epsilon\right)^2 \\
&\quad+\left\|R\right\|^2\left(\left|s_i^{(k)} \alpha_{k+1} e_{k+1}^{\mathrm{T}} p_i^{(k)}-c_i^{(k)} \bar{\beta}_k e_k^{\mathrm{T}} \bar{p}_i^{(k)}\right|+O\left(\epsilon\right)+O\left(\left\|\underline{B}_k^{-1}\right\|\|\widehat{B}_k^{-1}\|\epsilon\right)\right)^2\\
&=\left\|R\right\|^2\left(\left|s_i^{(k)} \alpha_{k+1} e_{k+1}^{\mathrm{T}} p_i^{(k)}- c_i^{(k)} \bar{\beta}_k e_k^{\mathrm{T}}\bar{p}_i^{(k)}\right|+O\left(\left\|\underline{B}_k^{-1}\right\|
\|\widehat{B}_k^{-1}\|\epsilon\right)\right)^2, 
\end{aligned}
$$
which establishes the first bound in \eqref{boundsfinite}.
The second bound in \eqref{boundsfinite} follows from 
$$
s_i^{(k)} \alpha_{k+1} e_{k+1}^{\mathrm{T}} p_i^{(k)}-c_i^{(k)} \bar{\beta}_k e_k^{\mathrm{T}} \bar{p}_i^{(k)}=\frac{\alpha_{k+1} \beta_{k+1}}{c_i^{(k)} s_i^{(k)}} e_k^{\mathrm{T}} w_i^{(k)},
$$
as the proof of Theorem \ref{Th:bound of residual} has already shown.
\end{proof}

We next define a relative residual. If an algorithm converges and is numerically backward stable, the defined relative residual norm ultimately achieves the level of $\epsilon$.
\begin{definition}
\label{Def:relative residual}
\begin{equation}
\label{relative residual}
relres_i^{(k)}:=\frac{1}{\left\|R\right\|}\left(\begin{array}{c}
A x_i^{(k)}-c_i^{(k)} y_i^{(k)} \\
L x_i^{(k)}-s_i^{(k)} z_i^{(k)} \\
s_i^{(k)} A^{\mathrm{T}} y_i^{(k)}-c_i^{(k)} L^{\mathrm{T}} z_i^{(k)}
\end{array}\right).
\end{equation}
\end{definition}
Theorem~\ref{Th:bound of residual} and
Theorem \ref{Th:bound of residual in finite} suggest a practical stopping criterion:
an approximate GSVD component $(c_i^{(k)},s_i^{(k)},y_i^{(k)},z_i^{(k)},x_i^{(k)})$ is claimed to
have converged if
\begin{equation}
\label{stopping criterion}
\left|s_i^{(k)} \alpha_{k+1} e_{k+1}^{\mathrm{T}} p_i^{(k)}-c_i^{(k)} \bar{\beta}_k e_k^{\mathrm{T}} \bar{p}_i^{(k)}\right| < tol
\end{equation}
or
\begin{equation}
\label{stopping criterion 2}
\left|\frac{\alpha_{k+1} \beta_{k+1}}{c_i^{(k)} s_i^{(k)}} e_k^{\mathrm{T}} w_i^{(k)}\right| < tol \end{equation}
for a prescribed tolerance $tol$.
Suppose the stopping tolerance $tol\geq \epsilon$, which is necessary in finite precision arithmetic. When $O\left(\left\|\underline{B}_k^{-1}\right\|\|\widehat{B}_k^{-1}\|\epsilon\right)\leqslant tol$, the left-hand side in (\ref{stopping criterion}) or (\ref{stopping criterion 2}) is a reliable estimate of the relative residual norm. Computationally, we benefit very much from using such stopping criteria since there is no need to compute $x_i^{(k)},\,y_i^{(k)}$ or $z_i^{(k)}$ explicitly until convergence.

However, when $\|\underline{B}_k^{-1}\|$ or $\|\widehat{B}_k^{-1}\|$ is so large that $O\left(\left\|\underline{B}_k^{-1}\right\|\|\widehat{B}_k^{-1}\|\epsilon\right)> tol$, the above stopping
criteria may not be reliable any longer because each of the first terms in bounds \eqref{boundsfinite} may
be smaller than the actual $\|r_i^{(k)}\|/\|R\|$. Jia and Li \cite{Jia-JBD-finite-precision} have analyzed
when $\|\underline{B}_k^{-1}\|$ or $\|\widehat{B}_k^{-1}\|$ is possibly large: Since the
eigenvalues of $\underline{B}_k^{\mathrm T}\underline{B}_k$ are the Ritz values of $Q_AQ_A^{\mathrm T}$
with respect to $\operatorname{span}(U_k)$, when $Q_A$ is not of full row rank, the smallest eigenvalue of
$\underline{B}_k^{\mathrm T}\underline{B}_k$ may approach the zero eigenvalue of $Q_AQ_A^{\mathrm T}$ as
$k$ increases. This causes that $\|\underline{B}_k^{-1}\|$ may grow uncontrollably. Similarly, the
eigenvalues of $\widehat{B}_k^{\mathrm T}\widehat{B}_k$ are the Ritz values of $Q_L^{\mathrm T}Q_L$ with
respect to $\operatorname{span}(V_k)$. When $Q_L$ is not of full column rank, $\|\widehat{B}_k^{-1}\|$ may
become large as $k$ increases. The analysis in \cite{Jia-JBD-finite-precision} shows that the necessary
condition for  $\|\underline{B}_k^{-1}\|$ and $\|\widehat{B}_k^{-1}\|$ being uniformly bounded is that
$m\leqslant n\leqslant p$ and $Q_A$ and $Q_L$ are of full row rank and column rank, respectively.
In this case, when $A$ or $L$ is ill conditioned, the smallest eigenvalue of $Q_AQ_A^{\mathrm T}$
or $Q_L^{\mathrm T}Q_L$ is near zero, so that $\|\underline{B}_k^{-1}\|$ or $\|\widehat{B}_k^{-1}\|$,
though uniformly bounded, may be large.

\section{Implicit restart of the JBD process}
In practical computations, the subspace size $k$ cannot be large due to the limitation of memory and computational cost.
However, a small $k$ generally cannot guarantee the convergence of the JBD method. Therefore, restarting is
necessary. In this section, we show how to adapt the implicit restart technique to the JBD method.
We first describe one-step joint implicit restart of Lanczos lower and upper bidiagonalizations of $Q_A$ and $Q_L$, and then show how to carry out one-step implicit restart of the JBD
process in the first two subsections, respectively. Finally, we show how to realize multi-step
implicit restart of the JBD process in the third subsection. We will consider some subtle issues
of effective and efficient implementations on the implicit restart of the JBD process in finite precision.

\subsection{One-step implicit restart of $Q_A$}
Let the nonnegative $\lambda_1\leqslant 1$ be a shift, and consider one implicit QR step on $B_kB_k^\mathrm{T}$.
Computationally, we implement this by working on $B_k$ directly without forming $B_kB_k^{\mathrm{T}}$
explicitly in order to avoid possible accuracy loss in finite precision. What we need is only the first column of $B_kB_k^{\mathrm{T}}$. First, find a Givens matrix
$G_1$ acting on rows 1 and 2 of $B_kB_k^{\mathrm T}-\lambda_1^2I_{k+1}$ such that the (2,1)-entry of
$G_1^{\mathrm T}\left(B_kB_k^{\mathrm T}-\lambda_1^2I_{k+1}\right)$ is zero. Second, find Givens matrices
$G_2,...,G_k$ and $P_1,...,P_{k-1}$ such that $G_k^{\mathrm{T}}\cdots G_1^{\mathrm{T}}B_kP_1\cdots P_{k-1}$ is
still a lower bidiagonal matrix. This is a bulge-chasing process, which is illustrated by the following
diagram with $k=4$, where ``$*$" denotes an original nonzero entry, ``$+$" is a nonzero entry newly
generated, and ``0" is zero after annihilation:

$$
G_1^{\mathrm T}\left(B_kB_k^{\mathrm T}-\lambda_1^2I_{k+1}\right)=\left(\begin{array}{ccccc}
* & * & + & & \\
0 & * & * & & \\
& * & * & * & \\
& & * & * & * \\
& & & * & *
\end{array}\right)
$$
$$
\resizebox{\textwidth}{!}{$
B_k=\left(\begin{array}{cccc}
* & & & \\
* & * & & \\
& * & * &\\
& & * & * \\
& & & *
\end{array}\right) \stackrel{G_1^{\mathrm{T}}=\operatorname{Givens}(1,2)}{\longrightarrow}\left(\begin{array}{cccc}
* & + & & \\
* & * & & \\
& * & * &\\
& & * & * \\
& & & *
\end{array}\right) \stackrel{P_1=\operatorname{Givens}(1,2)}{\longrightarrow}\left(\begin{array}{cccc}
* & 0 & & \\
* & * & & \\
+ & * & * &\\
& & * & * \\
& & & *
\end{array}\right) \\
$}
$$
$$
\stackrel{G_2^{\mathrm{T}}=\operatorname{Givens}(2,3)}{\longrightarrow}\left(\begin{array}{cccc}
* & & & \\
* & * & + & \\
0 & * & * &\\
& & * & * \\
& & & *
\end{array}\right) \stackrel{P_2=\operatorname{Givens}(2,3)}{\longrightarrow}\left(\begin{array}{cccc}
* & & & \\
* & * & 0 & \\
& * & * &\\
& +& * & * \\
& & & *
\end{array}\right) \stackrel{G_3^{\mathrm{T}}=\operatorname{Givens}(3,4)}{\longrightarrow} \\
$$
$$
\left(\begin{array}{cccc}
* & & & \\
* & * & & \\
& * & * & +\\
& 0 & * & * \\
& & & *
\end{array}\right) \stackrel{P_3=\operatorname{Givens}(3,4)}{\longrightarrow}\left(\begin{array}{cccc}
* & & & \\
* & * & & \\
& * & * & 0\\
& & * & * \\
& & + & *
\end{array}\right) \stackrel{G_4^{\mathrm{T}}=\operatorname{Givens}(4,5)}{\longrightarrow}\left(\begin{array}{cccc}
* & & & \\
* & * & & \\
& * & * &\\
& & * & * \\
& & 0 & *
\end{array}\right)
$$

The above process implicitly implements one-step QR iteration of $B_kB_k^{\mathrm{T}}$ with the shift $\lambda_1^2$ and changes $B_k$ to a new lower bidiagonal matrix $\dot{B}_k$. Let $G^{(1)}=G_1\cdots G_k$, $P^{(1)}=P_1\cdots P_{k-1}$. Then we have $\dot B_k=G^{(1)\mathrm T}B_kP^{(1)}$ and
\begin{equation}
\label{untruncated restarted lower bidiag}
Q_A \dot V_k=\dot U_{k+1} \dot B_k, \quad Q_A^{\mathrm T} \dot U_{k+1}=\dot V_k \dot B_k^{\mathrm T}+\alpha_{k+1} v_{k+1} e_{k+1}^{\mathrm T}G^{(1)},
\end{equation}
where $\dot V_k=V_kP^{(1)}$ and $\dot U_{k+1}=U_{k+1}G^{(1)}$. The columns of new basis matrices $\dot V_k$ and $\dot U_{k+1}$ are $\dot v_i=V_kP^{(1)}e_i$, $i=1,2,\dots,k$ and $\dot u_i=U_{k+1}G^{(1)}e_i$, $i=1,2,\dots,k+1$.

Let $\dot V_{k-1}$ and $\dot U_k$ be the matrices consisting of the first $k-1$ and $k$ columns of $\dot V_k$ and $\dot U_{k+1}$, respectively, and $\dot B_{k-1}$ be the leading $k\times (k-1)$ submatrix of $\dot B_k$, and define $\dot \alpha_k=e_k^{\mathrm{T}}\dot B_k e_k$. Note that $e_{k+1}^{\mathrm T}G^{(1)}=e_{k+1}^{\mathrm T}G_k$ has nonzero entries only on the last two positions. By truncating (\ref{untruncated restarted lower bidiag}), we obtain
\begin{equation}
\label{restarted lower bidiag}
Q_A \dot{V}_{k-1}=\dot U_{k} \dot B_{k-1}, \quad Q_A^{\mathrm T} \dot U_{k}=\dot V_{k-1} \dot B_{k-1}^{\mathrm T}+\dot r_{k-1}e_{k}^{\mathrm{T}},
\end{equation}
where
\begin{equation}
\label{dotrk-1}
\dot r_{k-1}=\alpha_{k+1}\left(e_{k+1}^{\mathrm{T}}G^{(1)}e_k\right)v_{k+1}+\dot\alpha_k \dot v_k.
\end{equation}

\begin{theorem}
\label{Th:restarted lower bidiag}
Relation \textnormal{(\ref{restarted lower bidiag})} is an implicitly restarted $(k-1)$-step Lanczos lower bidiagonalization process of $Q_A$, and the updated starting vector is
\begin{equation}
\label{expression of dot u1}
\dot u_1=\pm\frac{\left(Q_A Q_A^{\mathrm T}-\lambda_1^2 I_{m}\right) u_1}{||\left(Q_A Q_A^{\mathrm T}-\lambda_1^2 I_{m}\right) u_1||}.
\end{equation}
\end{theorem}

\begin{proof}
By $\dot{V}_k=V_kP^{(1)}$ and (\ref{dotrk-1}), we have $\dot V_{k-1}^{\mathrm{T}}\dot r_{k-1}=0$. This means that \eqref{restarted lower bidiag} is an $(k-1)$-step Lanczos lower bidiagonalization process
of $Q_A$.

According to the afore-described implicit QR iteration of $B_k$ with the nonnegative shift $\lambda_1$, we know that
$G^{(1)}$ is the Q-factor in the QR factorization of $B_kB_k^{\mathrm{T}}-\lambda_1^2I_{k+1}$.
Write $B_kB_k^{\mathrm T}-\lambda_1^2I_{k+1}=G^{(1)}R^{(1)}$. From the second relation in (\ref{relation to
Lanczos}), we obtain
$$
\begin{aligned}
\left(Q_A Q_A^{\mathrm T}-\lambda_1^2 I_{m}\right) U_{k+1} & =U_{k+1}\left(B_k B_k^{\mathrm T}-\lambda_1^2 I_{k+1}\right)+\alpha_{k+1} Q_{A} v_{k+1} e_{k+1}^{\mathrm T} \\
& =U_{k+1} G^{(1)} R^{(1)}+\alpha_{k+1} Q_A v_{k+1} e_{k+1}^{\mathrm T}.
\end{aligned}
$$
Taking the first columns in the two hand sides gives $\left(Q_A Q_A^{\mathrm T}-\lambda_1^2 I_{m}\right) u_1=\left(e_1^{\mathrm T}R^{(1)}e_1\right) \dot{u}_1$. Therefore,
$$
\dot u_1=\pm\frac{\left(Q_A Q_A^{\mathrm T}-\lambda_1^2 I_{m}\right) u_1}{\|\left(Q_A Q_A^{\mathrm T}-\lambda_1^2 I_{m}\right) u_1\|}.
$$

\end{proof}

We can use the above theorem to establish the following result, which is crucial to 
implicitly restart Lanczos bidiagonalization of $Q_L$ in 
the next subsection.

\begin{theorem}
\label{Th:one-step restarted Lanczos}
The afore-described implicit restart process implements one-step QR iteration of $B_k^{\mathrm{T}}B_k$
with the shift $\lambda_1$, the updated 
\begin{equation}
\label{expression of dot v1}
\dot v_1=\pm\frac{\left(Q_A^{\mathrm T} Q_A-\lambda_1^2 I_{n}\right) v_1}{\|\left(Q_A^{\mathrm T} Q_A-\lambda_1^2 I_{n}\right) v_1\|}, 
\end{equation}
and $P^{(1)}$ is the Q-factor in the QR factorization of $B_k^{\mathrm{T}}B_k-\lambda_1^2I_k$.
\end{theorem}

\begin{proof}
Taking the first columns in the second relation of (\ref{restarted lower bidiag}) and noticing that $\dot B_{k-1}^{\mathrm T}e_1$ only has nonzero entry in the first position, we obtain
$$
\dot v_1=\frac{Q_A^{\mathrm T}\dot u_1}{\|Q_A^{\mathrm T}\dot u_1\|}.
$$
Then (\ref{expression of dot v1}) holds because of (\ref{expression of dot u1}).
The first relation in (\ref{relation to Lanczos}) is the symmetric Lanczos process of $Q_A^{\mathrm{T}}Q_A$ with the starting vector $v_1$. From relationship (\ref{expression of dot v1}) between $\dot v_1$ and $v_1$, we have also proved that the described bulge-chasing process on $B_k$ implements one-step implicit QR iteration of $B_k^{\mathrm{T}}B_k$ with the shift $\lambda_1^2$. It follows from $\dot B_k^{\mathrm{T}}\dot B_k=P^{(1)\mathrm{T}}B_k^{\mathrm{T}}B_kP^{(1)}$ that the Q-factor in the QR factorization of $B_k^{\mathrm{T}}B_k-\lambda_1^2I_k$ is $P^{(1)}$.
\end{proof}

\subsection{One-step implicit restart of $Q_L$}
Let $\mu_1\geqslant 0$ satisfy $\lambda_1^2+\mu_1^2=1$, and consider one implicit QR step on $\bar B_k^\mathrm{T}\bar B_k$ with the shift $\mu_1^2$. We implement this process by working on $\bar B_k$ directly without forming $\bar B_k^{\mathrm{T}}\bar B_k$ explicitly, analogously to what we have done on $B_k$ in the last subsection.
$$
\resizebox{\textwidth}{!}{$
\begin{aligned}
&\bar P_1^{\mathrm T}\left(\bar B_k^{\mathrm T}\bar B_k-\mu_1^2I_{k}\right)=\left(\begin{array}{cccc}
* & * & + & \\
0 & * & * & \\
& * & * & * \\
& & * & *
\end{array}\right) \\
&\begin{aligned}
\bar B_k=&\left(\begin{array}{cccc}
* & * & & \\
& * & * & \\
& & * &* \\
& & & * \\
\end{array}\right) \stackrel{\bar P_1=\operatorname{Givens}(1,2)}{\longrightarrow}\left(\begin{array}{cccc}
* & * & & \\
+ & * & * & \\
& & * & * \\
& & & * \\
\end{array}\right) \stackrel{\bar G_1^{\mathrm{T}}=\operatorname{Givens}(1,2)}{\longrightarrow}\left(\begin{array}{cccc}
* & * & + & \\
0 & * & * & \\
& & * & * \\
& & & * \\
\end{array}\right) \\
& \stackrel{\bar P_2=\operatorname{Givens}(2,3)}{\longrightarrow}\left(\begin{array}{cccc}
* &* &0 & \\
 & * & * & \\
 & + & * &*\\
& & & * \\
\end{array}\right) \stackrel{\bar G_2^{\mathrm T}=\operatorname{Givens}(2,3)}{\longrightarrow}\left(\begin{array}{cccc}
* &* & & \\
 & * & * &+ \\
& 0 & * &*\\
& & & * \\
\end{array}\right) \stackrel{\bar P_3=\operatorname{Givens}(3,4)}{\longrightarrow} \\
& \left(\begin{array}{cccc}
* & *& & \\
 & * &* &0 \\
&  & * & *\\
& & +& * \\
\end{array}\right) \stackrel{\bar G_3^{\mathrm T}=\operatorname{Givens}(3,4)}{\longrightarrow}\left(\begin{array}{cccc}
* & *& & \\
& * & *& \\
&  & * & *\\
& & 0 & * \\
\end{array}\right)=:\widetilde{B}_k
\end{aligned}
\end{aligned}
$}
$$

Let $\bar P^{(1)}=\bar P_1\cdots \bar P_{k-1}$. Then $\bar P^{(1)}$ is the Q-factor in the
QR factorization of $\bar B_k^{\mathrm{T}}\bar B_k-\mu_1^2I_k$. From the key relationship (\ref{BTB+barBTbarB=I}) between $B_k$ and $\bar B_k$, we have
$$\bar B_k^{\mathrm{T}}\bar B_k-\mu_1^2I_k=-(B_k^{\mathrm{T}}B_k-\lambda_1^2I_k).$$
Therefore, the implicit QR iteration of $\bar B_k^{\mathrm{T}}\bar B_k-\mu_1^2I_k$ and that of $B_k^{\mathrm{T}}B_k-\lambda_1^2I_k$ are identical. By the implicit Q-theorem \cite[pp. 460]{GolubandVan2013matrix}, we obtain $\bar P_i=P_i$ for $i=1,\,2,\,\dots,k-1$, and thus $\bar P^{(1)}=P^{(1)}$. Let $\bar G^{(1)}=\bar G_1\cdots \bar G_{k-1}$, and $\check U_{k}=\widehat U_{k}\bar G^{(1)}$. Then $\widetilde B_k=\bar G^{(1)\mathrm T}\bar B_kP^{(1)}$. From \eqref{Lanczos bidiagonalization} and \eqref{bkbar}, we obtain
\begin{equation}
\label{untruncated restarted upper bidiag}
Q_L \dot{V}_k=\check{U}_k \widetilde{B}_k, \quad Q_L^{\mathrm T} \check{U}_k=\dot{V}_k \widetilde{B}_k^{\mathrm T}+\bar{\beta}_k v_{k+1} e_k^{\mathrm T} \bar G_k^{(1)},
\end{equation}
where $\dot{V}_k$ is the same as that in (\ref{untruncated restarted lower bidiag}) and $\breve{U}_k=(\breve{u}_1,\ldots,\breve{u}_k)$.

Let $\check U_{k-1}$ be the matrix consisting of the first $k-1$ columns of $\check U_k$. Let $\widetilde B_{k-1}$ be the leading $(k-1)\times (k-1)$ submatrix of $\widetilde B_k$ and $\tilde \beta_{k-1}=e_{k-1}^{\mathrm{T}}\widetilde B_k e_k$. Note that $e_{k}^{\mathrm T}\bar G^{(1)}=e_{k}^{\mathrm T}\bar G_{k-1}$ has nonzero entries only in the last two positions. Therefore, 
by truncating (\ref{untruncated restarted upper bidiag}), we obtain
\begin{equation}
\label{restarted upper bidiag}
Q_L \dot{V}_{k-1}=\check{U}_{k-1} \widetilde{B}_{k-1}, \quad Q_L^{\mathrm T} \check{U}_{k-1}=\dot{V}_{k-1} \widetilde{B}_{k-1}^{\mathrm T}+\tilde r_{k-1}e_{k-1}^{\mathrm{T}},
\end{equation}
where $\tilde r_{k-1}=\bar\beta_k\left(e_k^{\mathrm{T}}\bar G^{(1)}e_{k-1}\right)v_{k+1}+\tilde\beta_{k-1}\dot v_k$.

\begin{theorem}
Relation {\rm(\ref{restarted upper bidiag})} is the implicitly restarted $(k-1)$-step Lanczos upper bidiagonalization process of $Q_L$ with the nonnegative shift $\mu_1$ and the updated 
starting vector $\dot v_1$ defined as in \eqref{expression of dot v1}.
\end{theorem}

\begin{proof}
By $\dot{V}_k=V_kP^{(1)}$, it is straightforward that $\dot V_{k-1}^{\mathrm T}\tilde r_{k-1}=0$, which, together with the expression (\ref{expression of dot v1}) of $\dot v_1$, completes the proof.
\end{proof}

Since the orthonormal basis $\dot V_{k-1}$ of the right subspace in 
\eqref{restarted upper bidiag} after implicit restart is 
identical to that of $Q_A$ in (\ref{restarted lower bidiag}), the relations (\ref{restarted lower bidiag}) and (\ref{restarted upper bidiag}) carry out an implicitly restarted $(k-1)$-step joint Lanczos bidiagonalization processes of $Q_A$ and $Q_L$ with the shifts $\lambda_1$ and $\mu_1$, and the updated starting vector after restarting is $\dot u_1$.

\begin{remark}
Theoretically, exploiting the established relationship between the implicit restart of Lanczos lower and upper bidiagonalizations, there is no need to generate $P_i$ again when implicitly restarting the  Lanczos upper bidiagonalization process of $Q_L$ because the implicit restart of the Lanczos lower bidiagonalization process of $Q_A$ has generated $P_i,\ i=1,\dots,k-1$. Therefore, we can apply $P_i$ to $\bar B_k$ in turn directly, and at application step of $P_i$, we only need to calculate the Givens matrices $\bar G_i$ to annihilate those subdiagonal entries newly generated.
\end{remark}

\begin{remark}\label{finite}
In finite precision arithmetic,
due to rounding errors, when applying $P_i$ in the above way, those already
annihilated entries during the bulge-chasing process are now at the level of $\epsilon$ rather than exact zeroes, causing that the resulting matrix is not an exact upper bidiagonal matrix. To this end, a
computationally reliable and accurate way is to set the $O(\epsilon)$ entries that are not on the
diagonal and the upper-diagonal to zeroes, and thus obtains an upper bidiagonal matrix. Very importantly,
this approach guarantees that the orthonormal bases of the right subspaces generated by the two implicit
restart processes of Lanczos lower and upper bidiagonalizations are always identical and the JBD method
compute unique approximate right generalized singular vectors and, in the meantime, that the relations
of the Lanczos bidiagonalization process (\ref{restarted upper bidiag}) of $Q_L$ hold at the level of
$\epsilon$.
\end{remark}

Keep in mind the connection of \eqref{Lanczos bidiagonalization} and \eqref{JBD}, \eqref{BTB+barBTbarB=I}. 
The implicitly restarted Lanczos lower and upper bidiagonalization processes (\ref{restarted lower bidiag}) and (\ref{restarted upper bidiag}) 
lead to the implicitly restarted $(k-1)$-step \Cref{alg:JBD} with the updated starting vector $\dot u_1$:
\begin{equation}
\label{one-step restarted JBD}
    \begin{aligned}
    & \left(I_m, 0_{m,p}\right) \dot V_{k-1}^{\prime}=\dot U_{k} \dot B_{k-1}, \\
    & Q Q^{\mathrm T}\left(\begin{array}{c}
    \dot U_{k} \\
    0_{p,k}
    \end{array}\right)=\dot V_{k-1}^{\prime} \dot B_{k-1}^{\mathrm T}+\dot r_{k-1}^{\prime} e_{k+1}^{\mathrm T}, \\
    & \left(0_{p,m}, I_p\right) \dot V_{k-1}^{\prime} D_{k-1}=\check{U}_{k-1} \check{B}_{k-1},
    \end{aligned}
\end{equation}
where $\dot V_{k-1}^{\prime}=Q\dot V_{k-1}$, $\dot r_{k-1}^{\prime}=Q\dot r_{k-1}$, and $D_{k-1}=\operatorname{diag}(1,-1,\dots,(-1)^{k-2})\in\mathbb{R}^{k\times k}$, i.e., 
\begin{equation}
\label{AH=UBcheck}
\begin{aligned}
&    A \dot H_{k-1}=\dot U_{k} \dot B_{k-1},\quad L\dot H_{k-1}=\check{U}_{k-1} \check{B}_{k-1},\\
&    \dot B_{k-1}^{\mathrm T} \dot B_{k-1}+\check{B}_{k-1}^{\mathrm T} \check{B}_{k-1}=I_{k-1}.
\end{aligned}
\end{equation}
where $\dot H_{k-1}=R^{-1}\dot V_{k-1}$.

\subsection{Multi-step implicit restart of the JBD process}
\label{sec:implicit restart}
Given integers $k$ and $l$ with $l<k$, we next
consider $(k-l)$-step implicit restarts of the JBD process. For the $k-l$ pairs $\{\lambda_i^2$, $\mu_i^2\}$ of shifts satisfying $\lambda_i^2+\mu_i^2=1$, $i=1,2,\ldots,k-l$, we can do $(k-l)$ implicit QR steps on $B_kB_k^{\mathrm T}$ and $\bar B_k^{\mathrm T}\bar B_k$ with the shifts $\lambda_1^2,\dots,\,\lambda_{k-l}^2$ and $\mu_1^2,\dots,\,\mu_{k-l}^2$ in turn in the preceding way, respectively. Then we obtain the orthogonal matrices $\{G^{(i)}\}_{i=1}^{k-l},\,\{P^{(i)}\}_{i=1}^{k-l}$, and $\{\bar G^{(i)}\}_{i=1}^{k-l}$. Define \begin{equation}
\label{G,Gbar,P}
G=G^{(1)}\cdots G^{(k-l)},\,P=P^{(1)}\cdots P^{(k-l)},\,\bar G=\bar G^{(1)}\cdots \bar G^{(k-l)}
\end{equation}
and
\begin{equation}\label{updated}
B_k^+=G^{\mathrm T}B_kP,\ \ V_k^+=V_kP,\,U_{k+1}^+=U_{k+1}G,\ \ \bar B_k^+=\bar G^{\mathrm T}\bar B_k P,\ \ \widehat U_k^+=\widehat U_k\bar G.
\end{equation}
Then from \eqref{Lanczos bidiagonalization} and \eqref{bkbar} we obtain
\begin{equation}
\label{untruncated restarted joint bidiag}
    \begin{aligned}
    & Q_A V_k^+=U_{k+1}^+ B_k^+, \quad Q_A^{\mathrm T} U_{k+1}^+=V_k^+ B_k^{+\mathrm T}+\alpha_{k+1} v_{k+1} e_{k+1}^{\mathrm T} G, \\
    & Q_L V_k^+=\widehat{U}_k^+ \bar{B}_k^+, \quad Q_L^{\mathrm T} \widehat{U}_k^+=V_k ^+\bar{B}_k^{+\mathrm T}+\bar{\beta}_k v_{k+1} e_k^{\mathrm T}\bar G,
    \end{aligned}
\end{equation}
where $B_k^+$ and $\bar B_k^+$ are lower and upper bidiagonal, respectively, and still
satisfy the identity
$$
B_k^{+\mathrm T}B_k^++\bar B_k^{+\mathrm T}\bar B_k^+=I_k.
$$

Based on the above, we can obtain the implicitly restarted $l$-step JBD (IRJBD) process, as the following theorem shows.
\begin{theorem}
Let $V_l^+,\,U_{l+1}^+$ and $\widehat U_l^+$ be the matrices consisting of the first $l,\,l+1$ and $l$ columns of $V_k^+,\,U_{k+1}^+$ and $\widehat U_k^+$, and $B_{l}^+$ and $\bar B_{l}^+$  the leading $(l+1)\times l$ and $l\times l$ submatrices of $B_k^+$ and $\bar B_k^+$ in \eqref{updated}, respectively. Define $\alpha_{l+1}^+=e_{l+1}^{\mathrm T}B_k^+e_{l+1},\,\bar\beta_l^+=e_l^{\mathrm T}\bar B_k^+e_{l+1}$. Then there is the implicitly restarted $l$-step joint Lanczos bidiagonalization processes of $Q_A$ and Q$_L$ with the nonnegative shifts $\{\lambda_i\}_{i=1}^{k-l}$ and $\{\mu_i\}_{i=1}^{k-l}$ satisfying $\lambda_i^2+\mu_i^2=1$:
\begin{equation}
\label{restarted joint bidiag}
\begin{aligned}
& Q_A V_l^{+}=U_{l+1}^{+} B_l^{+}, \quad Q_A^{\mathrm T} U_{l+1}^{+}=V_l^{+} B_l^{+\mathrm T}+r_l^+ e_{l+1}^{\mathrm T}, \\
& Q_L V_l^{+}=\widehat{U}_l^{+} \widehat{B}_l^{+}, \quad Q_L^{\mathrm T} \widehat{U}_l^{+}=V_l^{+} \bar{B}_l^{+\mathrm T}+\bar r_l^+ e_l^{\mathrm T}, \\
\end{aligned}
\end{equation}
where 
\begin{equation}\label{quanupdate}
r_l^{+}=\alpha_{k+1} \left(e_{k+1}^{\mathrm T}Ge_{l+1}\right) v_{k+1}+\alpha_{l+1}^{+} v_{l+1}^{+},\ \  \bar r_l^+=\bar{\beta}_k \left(e_k^{\mathrm T}\bar Ge_l\right)v_{k+1}+\bar{\beta}_l^{+} v_{l+1}^{+},
\end{equation}
and the updated starting vectors $u_1^+=U_{l+1}^+e_1$ and $v_1^+=V_l^+e_1$ satisfy
\begin{equation}
\label{initial vector}
\begin{aligned}
&\tau u_1^+=\left(Q_A Q_A^{\mathrm T}-\lambda_1^2I_m\right)\cdots\left(Q_AQ_A^{\mathrm T}-\lambda_{k-l}^2I_m\right)u_1,\\
&\rho v_1^+=\left(Q_A^{\mathrm T}Q_A-\lambda_1^2I_n\right)\cdots\left(Q_A^{\mathrm T}Q_A-\lambda_{k-l}^2I_n\right)v_1
\end{aligned}
\end{equation}
with $\tau$ and $\rho$ being normalizing factors.
\end{theorem}

\begin{proof}
Note that $e_{k+1}^{\mathrm T}G$ and $e_k^{\mathrm T}\bar G$ has nonzero entries only in the last $k-l+1$ positions. Relation (\ref{restarted joint bidiag}) follows by truncating (\ref{untruncated restarted joint bidiag}). By $V_k^+=V_kP$, it is straightforward that $r_l^+$ and $\bar r_l^+$ are orthogonal to $V_l^+$. Therefore, (\ref{restarted joint bidiag}) is the implicitly restarted $l$-step joint bidiagonalization processes of $Q_A$ and $Q_L$ with the starting vectors $u_1^+$ and $v_1^+$, respectively.

Next we prove (\ref{initial vector}) by induction. Theorem \ref{Th:restarted lower bidiag} has proved the case $k-l=1$. Suppose the relation holds for $k-l=i$:
$$
\tau_i u_1^{(i)}=\left(Q_A Q_A^{\mathrm T}-\lambda_1^2I_m\right)\cdots\left(Q_AQ_A^{\mathrm T}-\lambda_i^2I_m\right)u_1,\\
$$
where $u_1^{(i)}=U_{k+1}G^{(1)}\cdots G^{(i)}e_1$ and $\tau_i$ is a normalizing factor. Define $u_1^{(i+1)}=U_{k+1}G^{(1)}\cdots G^{(i+1)}e_1$. By Theorem \ref{Th:restarted lower bidiag}, we have
$$
\begin{aligned}
u_1^{(i+1)}&=\gamma \left(Q_A Q_A^{\mathrm{T}}-\lambda_{i+1}^2 I_m\right) u_1^{(i)} \\
&=\gamma^{\prime}\left(Q_A Q_A^{\mathrm T}-\lambda_1^2I_m\right)\cdots\left(Q_AQ_A^{\mathrm T}-\lambda_{i+1}^2I_m\right)u_1,
\end{aligned}
$$
where $\gamma$ and $\gamma^{\prime}$ are normalizing factors, proving the
first relation of (\ref{initial vector}). On the other hand,  
\begin{eqnarray*}
v_1^+&=&\frac{Q_A^{\mathrm T}u_1^+}{\|Q_A^{\mathrm T}u_1^+\|}\\
&=&\frac{1}{\tau\|Q_A^{\mathrm T}u_1^+\|}\left(Q_A^{\mathrm T}Q_A-\lambda_1^2I_n\right)\cdots\left(Q_A^{\mathrm T}Q_A-\lambda_{k-l}^2I_n\right)Q_A^{\mathrm T}u_1\\
&=&\frac{1}{\rho}\left(Q_A^{\mathrm T}Q_A-\lambda_1^2I_n\right)\cdots\left(Q_A^{\mathrm T}Q_A-\lambda_{k-l}^2I_n\right)v_1. \hspace{2cm} \Box
\end{eqnarray*}

\end{proof}

Based on the relationships \eqref{Lanczos bidiagonalization} and \eqref{BTB+barBTbarB=I}
between two Lanczos bidiagonalizations and the JBD process, by \eqref{restarted joint bidiag}, we have come to our ultimate result of this section.
\begin{theorem}{restarted joint bidiag}
Let $H_l^+=R^{-1}V_l^+$. Then the preceding $(k-l)$-step implicit restart of joint
Lanczos lower and upper bidiagonalizations generates an implicitly restarted $l$-step JBD process:
\begin{equation}
\begin{aligned}
&AH_l^+=U_{l+1}^+B_l^+,\quad LH_l^+=\widehat U_l^+\bar B_l^+,\\
&B_l^{+\mathrm T}B_l^++\bar B_l^{+\mathrm T}\bar B_l^+=I_l.
\end{aligned}
\end{equation}
\end{theorem}

This theorem forms the basis of developing a
practical implicitly restarted JBD algorithm (IRJBD) for computing several
extreme GSVD components of the matrix pair $\{A,\,L\}$. We remind that,
in finite precision arithmetic, we implement the IRJBD process for each shift in the way
that Remark~\ref{finite} states.

\subsection{A comparison with the thick restart JBD algorithm}\label{sec:compare}
In \cite{TRJBD}, a thick restart JBD algorithm is proposed. We briefly review it and compare its computational cost with IRJBD. Define $P_{k+1}=(P_k,p_{k+1})$, where $P_k$ is defined in (\ref{GSVD of Bk and Bkbar}) and $p_{k+1}$ is a vector such that $p_{k+1}^{\mathrm T}P_k=0$. Then $P_{k+1}$ is an orthogonal matrix. Combining GSVD (\ref{GSVD of Bk and Bkbar}) of $\{B_k,\bar B_k\}$ and relation (\ref{Lanczos bidiagonalization}), we obtain
$$
\resizebox{\textwidth}{!}{$
\begin{array}{ll}
Q_A V_k W_k=U_{k+1} P_{k+1} \left(\begin{array}{l}
\Theta_k \\
0
\end{array}\right), & Q_A^{\mathrm T} U_{k+1} P_{k+1}=V_k W_k(\Theta_k,0)+\alpha_{k+1} v_{k+1} e_{k+1}^{\mathrm T} P_{k+1}, \\
Q_L V_k W_k=\widehat{U}_k \bar{P}_k \Psi_{k}, & Q_L^{\mathrm T} \widehat{U}_k \bar{P}_k=V_k W_k \Psi_k+\bar\beta_k v_{k+1} e_k^{\mathrm T} \bar{P}_k.
\end{array}
$}
$$
Let $W_l$ and $\bar P_l$ be the matrices composed of the first $l$ columns of $W_k$ and $\bar P_k$, respectively, and $P_{l,k+1}=(p_1,\dots,p_l,p_{k+1})$, whose first $l$ columns are the first ones
of $P_k$. Let $\Theta_l$ and $\Psi_l$ be the $l\times l$ leading principal submatrices of $\Theta_k$ and $\Psi_k$, which generate the $l$ largest or smallest 
Ritz values and use them to
approximate $l$ largest or smallest generalized singular values of $\{A,L\}$. Denote
\begin{equation}
\label{updating orthonormal matrices in TRJBD}
U_{l,\rm{thick}}=U_{k+1} P_{l,k+1},\quad \widehat{U}_{l,\rm{thick}}=\widehat{U}_k \bar{P}_l,\quad V_{l,\rm{thick}}=V_k W_l.
\end{equation}
Then
$$
\begin{array}{ll}
Q_A V_{l,\rm{thick}}=U_{l,\rm{thick}}\left(\begin{array}{l}
\Theta_l \\
0
\end{array}\right), & Q_A^{\mathrm T} U_{l,\rm{thick}}=V_{l,\rm{thick}}(\Theta_l, 0)+v_{k+1}b_{l+1,\rm{thick}}^{\mathrm T}, \\
Q_L V_{l,\rm{thick}}=\widehat{U}_{l,\rm{thick}}\Psi_{l}, & Q_L^{\mathrm T} \widehat{U}_{l,\rm{thick}}=V_{l,\rm{thick}} \Psi_l+v_{k+1}\hat b_{l,\rm{thick}}^{\mathrm T},
\end{array}
$$
where $b_{l+1,\rm{thick}}=\alpha_{k+1}P_{l+1}^{\mathrm T}e_{k+1}$ and $\hat b_{l,\rm{thick}}=\bar\beta_k\bar P_l^{\mathrm T}e_{k}$. The one right subspace and two left subspaces for solving the GSVD problem of $\{Q_A,Q_L\}$ are the $l$-dimensional 
$\operatorname{span} (V_{l,\rm{thick}})$ and  $\operatorname{span}(U_{l,\rm{thick}})$, $\operatorname{span}(\widehat{U}_{l,\rm{thick}})$, respectively. These transformations reduce the subspace dimension from $k$ to $l$ while retaining all the information of the approximating 
$l$ right and left Ritz vectors
in the restarted subspaces. This is the $l$-step thick restart JBD process. Then one expands the subspaces through the JBD process, i.e., Algorithm \ref{alg:JBD}, from the $(l+1)$th step onwards to step $k$, and computes new approximations to the $l$ desired GSVD components with respect to the updated subspaces. Proceed in such a way until convergence. This is the thick restart JBD (TRJBD) algorithm proposed in \cite{TRJBD}. 

Regarding the computational cost of each restart, since $l$ and $k$ are very small compared to $m$, $p$ 
and $n$, the cost of updating $B_k$ and $\widehat B_k$ is negligible in both IRJBD and TRJBD. 
Except that the same number of large least squares problems are solved accurately, the other 
main computational cost comes from updating the three orthonormal basis matrices $U_{k+1}$, $\widehat 
U_k$, and $V_k$. For IRJBD, each of $G^{(i)}$, $\bar G^{(i)}$ and $P^{(i)}$ in (\ref{G,Gbar,P}) is a 
product of Givens transformations, and can be proved to be an upper Hessenberg matrix. Thus, $G$, $\bar G$ 
and $P$ have the property that every entry at position $(i,\,j)$ satisfying $i-j>k-l$ is zero. Then, the 
cost of computing $U_{l+1}^+$, $\widehat U_l^+$ and $QV_l^+$ is $2(m+p)(2kl-l^2)$ flops, ignoring 
low-order terms. However, for TRJBD, $P_{l,k+1}$, $\bar P_l$ and $W_l$ in (\ref{updating orthonormal 
matrices in TRJBD}) are full matrices, and the cost of computing $U_{l,\rm{thick}}$, $\widehat U_{l,\rm{thick}}$ and $QV_{l,\rm{thick}}$m is $4(m+p)kl$ flops, ignoring low-order terms. Thus, the 
overall computational cost of one restart of IRJBD is slightly lower than that of TRJBD.

\section{Selection of shifts} \label{selectshift}

Once nonnegative shifts $\{\lambda_i\}_{i=1}^{k-l}$ and $\{\mu_i\}_{i=1}^{k-l}$ satisfying
$\lambda_i^2+\mu_i^2=1$ are given, we can implicitly restart the JBD process as described previously.
However, the selection of shifts is crucial to make IRJBD converge as fast as possible for computing extreme GSVD components of $\{A,\,L\}$. In this section, we propose a selection strategy of shifts that are best possible
within the JBD method in some sense.
Because of $\lambda_i^2+\mu_i^2=1$, we only need to focus on selection of 
$\{\lambda_i\}_{i=1}^{k-l}$.

\subsection{Exact shifts}
For an implicitly restarted Krylov subspace algorithm for the eigenvalue problem, it has been proved in
\cite{Jia-poly-refined,Jia-refined-harmonic-Arnoldi} that the closer the shifts are to those undesired eigenvalues, the more
information the restarted Krylov subspace contains on the desired eigenvectors, so that the algorithm
should converge faster.
Jia and Niu \cite{Jia-IRRBL} extends these results to implicitly restarted Lanczos
bidiagonalization type algorithms for the large SVD problem; see Theorem 5.1 there.
These results are directly adaptable to IRJBD because it is mathematically equivalent to the implicitly restarted Lanczos bidiagonalization method for the SVD problems of $Q_A$ and $Q_L$.
It is seen from (\ref{initial vector}) that the better the shifts $\lambda_i$ 
approximate some of the undesired singular values of $Q_A$, the smaller the components of $u_1^+$ and $v_1^+$ are in the directions of the undesired left and right singular vectors of $Q_A$, respectively. 
Consequently, the column spaces of $U_{l+1}^+$ and $V_l^+$ contain more accurate approximations to the desired left and right singular vectors, so that the implicitly restarted 
Lanczos bidiagonalization methods for $Q_A$ and $Q_L$ and thus IRJBD generally converges faster. These results form the mechanism of selecting the best possible shifts within the framework of the JBD method.

In a $k$-step JBD method, the singular values $\{c_i^{(k)}\}_{i=1}^k$ of $B_k$ are the best available approximations to some of the singular values of $Q_A$. If we want to compute the $l$ largest generalized singular values of $\{A,\,L\}$ and the corresponding generalized left and right singular vectors, we select $\{c_i^{(k)}\}_{i=1}^l$ to approximate the $l$ largest singular values of $Q_A$, which correspond to the $l$ largest generalized singular values of $\{A,\,L\}$. Naturally, we use those undesired ones $\{c_i^{(k)}\}_{i=l+1}^k$ as shifts, called exact shifts. If we are interested in the $l$ smallest generalized singular values of $\{A,\,L\}$ and the corresponding generalized left and right singular vectors, we use $\{c_i^{(k)}\}_{i=k-l+1}^k$ to approximate the $l$ smallest singular values of $Q_A$, which correspond to the $l$ smallest generalized singular values of $\{A,\,L\}$, and select $\{c_i^{(k)}\}_{i=1}^{k-l}$ as exact shifts. The corresponding Ritz values $\{c_i^{(k)}/s_i^{(k)}\}_{i=1}^{k-l}$ 
are the best approximations obtained by the JBD method to some unwanted generalized singular values 
of $\{A,L\}$. 

\subsection{Adaptive shifting strategy}
If a shift $\lambda_i$ is close to some {\em desired} singular value $c_j$ of $Q_A$, (\ref{initial vector}) indicates that $u_1^+$ and $v_1^+$ will nearly annihilate the components of the left and right singular vectors of $Q_A$ corresponding to $c_j$, causing that the convergence becomes slow or even stagnates. This often occurs when the $l$th singular value is close to the $(l+1)$th one of $Q_A$. In \cite{Jia-IRRBL,Jia-IRRHLB}, Jia and Niu develop an adaptive shifting strategy to overcome this deficiency. We adapt their strategy to IRJBD.

For the computation of the $l$ largest generalized singular values, we define an approximation of the relative gap between the $l$th largest singular value $c_l$ of $Q_A$ and a shift $\lambda_i$ as
\begin{equation}
\label{gap of cl(k) and lambdai}
\operatorname{relgap}_{l,i}=\left|\frac{c_l^{(k)}-\lambda_i}{c_l^{(k)}}\right|.
\end{equation}
If $\operatorname{relgap}_{l,i}<10^{-3}$, we claim $\lambda_i$ to be a bad shift and set it to zero. Zero shifts avoid annihilating the components of the singular vectors corresponding to $c_l$ in $u_1^+$ and $v_1^+$ and thus overcome the drawback of the exact shifts.

If we are required to calculate the $l$ smallest generalized singular values, the strategy changes $l$ in (\ref{gap of cl(k) and lambdai}) to $k-l+1$. We set those $\lambda_i$ satisfying  $\operatorname{relgap}_{k-l+1,i}<10^{-3}$ to 1, the maximal possible shift, and fix 
the drawback of the exact shifts.

\section{The IRJBD algorithm and some details}
For large matrix eigenproblems and SVD problems, in order to speed up convergence, similar to 
the implicitly restarted Arnoldi algorithm, which is the function {\sf eigs} in MATLAB, Jia \cite{Jia-poly-refined} and Jia \& Niu \cite{Jia-IRRHLB} compute $l+3$ approximate eigenpairs or singular triplets, so the number of shifts is $k-(l+3)$ when $l$ eigenpairs or singular triplets are desired. We adapt this strategy to the IRJBD algorithm for GSVD problems. The default parameter $adjust=3$ means that $k-(l+3)$ shifts are used for implicit restart and the basis matrices of the restarted subspaces are $V_{l+3}^{+}$, $U_{l+4}^{+}$ and $\widehat U_{l+3}^{+}$. Algorithm \ref{alg:IRJBD} describes the IRJBD algorithm, and Table \ref{tab:parameters} lists the parameters and their default values.

\begin{algorithm}[ht]
	\renewcommand{\algorithmicrequire}{\textbf{Input:}}
	\renewcommand{\algorithmicensure}{\textbf{Output:}}
	\caption{The IRJBD algorithm}
	\label{alg:IRJBD}
	\begin{algorithmic}[1]
        \REQUIRE The matrices $A \in \mathbb{R}^{m \times n}$, $L \in \mathbb{R}^{p \times n}$, vector $u_1 \in \mathbb{R}^m$ with $\|u_1\|=1$, the number $l=|target|$ of the desired extreme GSVD components with $target>0$ or $target<0$ indicating that the largest or
        smallest GSVD components are required, $k_{\max}$, $adjust$, and
        the stopping tolerance $tol$.
         \ENSURE The $l$ converged GSVD components.
        \STATE Do the $k_{\max}$-step JBD process of $\{A,\,L\}$ using Algorithm \ref{alg:JBD}.
        \WHILE{the $l$ approximate GSVD components do not converge or restarts $\leq maxit$}
        \label{determining convergence in alg implementation}
            \IF{the basis size $<k_{\max}$}
                \STATE Do one step of the JBD process (Lines \ref{JBD inner start}-\ref{JBD inner end} of Algorithm \ref{alg:JBD}).
            \ELSE
                \STATE Compute the GSVD of $\{B_{k_{\max}},\,\bar{B}_{k_{\max}}\}$ and select the shifts $\lambda_i$, $i=1,2,\dots,k_{\max}-(l+adjust)$ according to the sign of $target$.
                \STATE Implicitly restart the JBD process as in section \ref{sec:implicit restart}.
            \ENDIF
        \ENDWHILE
        \STATE Compute the $l$ converged GSVD components according to section \ref{sec:JBD method}.
	\end{algorithmic}
\end{algorithm}

Some implementation details of the algorithm are as follows. The least squares problem (\ref{least-squares}) needs to be solved accurately to ensure that (\ref{BTB+barBTbarB=I}) holds at the level of $\epsilon$. We solve it by the MATLAB function {\sf lsqr} with the stopping tolerance $10\epsilon$. As for the maximum number of iterations, theoretically, the LSQR algorithm obtains the accurate solution after at most $n$ steps. However, in finite precision arithmetic, when the condition number of $(A^{\mathrm T},\,L^{\mathrm T})^{\mathrm T}$ is large, LSQR may not converge within $n$ steps due to the loss of orthogonality of basis vectors. Thus we set the maximum iteration steps to $10n$. We have observed that for the test problems in our experiments, $10n$ is enough to ensure convergence. We perform full reorthogonalization on $V_k^{\prime},\,U_k$ and $\widehat U_k$ to ensure that they are numerically orthonormal to the working precision $\epsilon$. By default, the convergence of line \ref{determining convergence in alg implementation} in Algorithm \ref{alg:IRJBD} is judged by (\ref{stopping criterion}) with the stopping tolerance $tol=10^{-8}$. The starting vector $u_1$ is randomly generated by the standard normal distribution function {\sf randn} in MATLAB and then normalized.

\begin{table}[ht]
\centering
\begin{tabularx}{\textwidth}{>{\raggedright\arraybackslash}p{0.15\textwidth}%
                                    >{\raggedright\arraybackslash}p{0.25\textwidth}%
                                    >{\raggedright\arraybackslash}p{0.5\textwidth}}
\hline
Parameters & Default values & Description \\ \hline
$target$ & 5 & $|target|=l$ is the number of desired GSVD components, where $target>0$ or $target<0$ means
that the largest or smallest GSVD components are required. \\ \hline
$k_{\max}$ &  & Maximum subspace dimension \\ \hline
$adjust$ & $3$ & Integer added to $l$ to speed up convergence \\ \hline
$tol$ & $10^{-8}$ & Stopping tolerance \\ \hline
$maxit$ & $1000$ & Maximum number of restarts \\ \hline
$lsqrtol$ & $10\epsilon$ & Stopping tolerance of {\sf lsqr} \\ \hline
$lsqrmaxit$ & $10n$ & Maximum number of iterations of {\sf lsqr} \\ \hline
$u_1$ & generated by {\sf randn} and then normalized & The unit length starting vector \\ \hline
\end{tabularx}
\caption{Parameters of IRJBD}
\label{tab:parameters}
\end{table}

\section{Numerical experiments}
We now present a number of numerical experiments to demonstrate the performance of Algorithm \ref{alg:IRJBD}, written in MATLAB language, and compare it with the public 
C language code of TRJBD in \cite{TRJBD}. The experiments will illustrate that
\Cref{alg:IRJBD} is at least competitive with and can outperform TRJBD considerably in terms of restarts.
Since section~\ref{sec:compare} has shown that
each restart of IRJBD and TRJBD  basically costs the same for the same $l$ and $k_{\max}$, fewer restarts mean higher efficiency. The experiments were performed on an AMD Ryzen 7 5800X CPU with 78 GB RAM and 8 cores using the MATLAB R2024a with the machine precision $\epsilon=2.22\times10^{-16}$ under the Ubuntu 20.04.3 LTS 64-bit system.

We use some matrices from the SuiteSparse Matrix Collection \cite{davis2011university} as $A$, and $L$ is taken as (cf. \cite{Huang-Numerical-experiments}):
$$
L=\left(\begin{array}{cccc}
3 & 1 & & \\
1 & \ddots & \ddots & \\
& \ddots & \ddots & 1 \\
& & 1 & 3
\end{array}\right),
$$
which is well conditioned. Table \ref{tab:matrices} lists the test matrices $A$ and some of their properties, where $nnz$ is the total number of nonzero entries of $\{A,L\}$ and $\kappa(A)=\frac{\sigma_{\max}(A)}{\sigma_{\text{min}}(A)}$, available from the \href{https://sparse.tamu.edu/}{SuiteSparse Matrix Collection}, is the condition number of $A$ and
indicates the level of ill-conditioning of $A$ and reflects magnitudes of generalized singular 
values of $\{A,L\}$. In order to ensure that all the $A$ are flat or square, i.e., $m\leqslant n$, we transpose some matrices, denoted by superscript $\mathrm T$. Note that the smallest singular value of kneser\_8\_3\_1$^\mathrm T$ is tiny, which causes that the matrix is numerically
row rank deficient and the smallest generalized singular for this case is tiny.

\begin{table}[H]
\centering
\begin{tabular}{|l|r|r|r|r|r|}
\hline
$A$             & $m$   & $n$   & $p$   & $nnz$  & $\kappa(A)$ \\ \hline
flower\_5\_4    & 5226  & 14721 & 14721 & 88103  & 14.9100     \\ \hline
kneser\_8\_3\_1$^{\mathrm T}$ & 15681 & 15737 & 15737 & 94251  & 6.40E+118   \\ \hline
l30             & 2701  & 16281 & 16281 & 100911 & 2.30E+03    \\ \hline
lp22            & 2958  & 16392 & 16392 & 117692 & 25.7823     \\ \hline
lp\_maros\_r7   & 3136  & 9408  & 9408  & 173070 & 2.3231      \\ \hline
model8          & 2896  & 6464  & 6464  & 44667  & 53.6287     \\ \hline
rosen10         & 2056  & 6152  & 6152  & 82646  & 194.0816    \\ \hline
\end{tabular}
\caption{Test matrices}
\label{tab:matrices}
\end{table}

Table \ref{tab:results} reports the results obtained by the IRJBD algorithm, where
$iter$ is the number of restarts, $Res_b$ is the largest relative residual bound (\ref{stopping criterion}) of the $l$ approximate GSVD components that approximate the desired extreme ones when the algorithm terminates, $Res$ is the actual largest relative residual norm and the quantity $\|\underline{B}_k^{-1}\|\|\widehat B_k^{-1}\|$ is its value when the algorithm stops. The norms $\|A\|,\,\|L\|$ and $\|R\|$ are estimated by $\sqrt{\|A\|_1\|A\|_{\infty}},\,\sqrt{\|L\|_1\|L\|_{\infty}}$ and the right-hand side of (\ref{bound of R}), respectively.

\begin{table}[]
\centering
\resizebox{\textwidth}{!}{
\begin{tabular}{|l|c|l|l|l|l|l|l|}
\hline
$A$ & $m\times n$ & $target$ & $k_{\max}$ & $iter$ & $Res_b$ & $Res$ & $\|\underline{B}_k^{-1}\|\|\widehat B_k^{-1}\|$ \\ \hline
\multirow{8}{*}{flower\_5\_4} & \multirow{8}{*}{\begin{tabular}[c]{@{}c@{}}5226\\ $\times$\\ 14721\end{tabular}} & \multirow{2}{*}{5} & 25 & 17 & 9.42E-09 & 4.55E-09 & 2.46E+01 \\ \cline{4-8}
 &  &  & 50 & 6 & 9.94E-09 & 5.40E-09 & 2.47E+01 \\ \cline{3-8}
 &  & \multirow{2}{*}{-5} & 25 & 14 & 9.37E-09 & 4.49E-09 & 2.34E+01 \\ \cline{4-8}
 &  &  & 50 & 5 & 9.64E-09 & 5.17E-09 & 2.49E+01 \\ \cline{3-8}
 &  & \multirow{2}{*}{10} & 25 & 38 & 9.16E-09 & 4.33E-09 & 2.32E+01 \\ \cline{4-8}
 &  &  & 50 & 10 & 9.61E-09 & 5.16E-09 & 2.45E+01 \\ \cline{3-8}
 &  & \multirow{2}{*}{-10} & 25 & 28 & 9.76E-09 & 5.15E-09 & 2.12E+01 \\ \cline{4-8}
 &  &  & 50 & 8 & 9.46E-09 & 4.48E-09 & 2.11E+01 \\ \hline
\multirow{8}{*}{kneser\_8\_3\_1} & \multirow{8}{*}{\begin{tabular}[c]{@{}c@{}}15681\\ $\times$\\ 15737\end{tabular}} & \multirow{2}{*}{5} & 25 & 44 & 9.96E-09 & 4.49E-06 & 2.32E+14 \\ \cline{4-8}
 &  &  & 50 & 13 & 8.10E-09 & 2.11E-05 & 1.50E+16 \\ \cline{3-8}
 &  & \multirow{2}{*}{-5} & 25 & 110 & 9.35E-09 & 2.50E-02 & 1.78E+63 \\ \cline{4-8}
 &  &  & 50 & 36 & 8.74E-09 & 1.99E-02 & 1.80E+61 \\ \cline{3-8}
 &  & \multirow{2}{*}{10} & 25 & 83 & 9.89E-09 & 5.70E-06 & 7.61E+14 \\ \cline{4-8}
 &  &  & 50 & 18 & 7.59E-09 & 5.09E-05 & 1.18E+16 \\ \cline{3-8}
 &  & \multirow{2}{*}{-10} & 25 & 441 & 9.69E-09 & 2.75E-02 & 1.66E+130 \\ \cline{4-8}
 &  &  & 50 & 59 & 8.92E-09 & 1.55E-02 & 6.05E+86 \\ \hline
\multirow{8}{*}{l30} & \multirow{8}{*}{\begin{tabular}[c]{@{}c@{}}2701\\ $\times$\\ 16281\end{tabular}} & \multirow{2}{*}{5} & 25 & 264 & 9.68E-09 & 6.21E-09 & 3.65E+03 \\ \cline{4-8}
 &  &  & 50 & 48 & 9.73E-09 & 6.23E-09 & 1.93E+03 \\ \cline{3-8}
 &  & \multirow{2}{*}{-5} & 25 & 63 & 9.53E-09 & 2.62E-09 & 4.59E+03 \\ \cline{4-8}
 &  &  & 50 & 14 & 9.28E-09 & 3.19E-09 & 4.64E+03 \\ \cline{3-8}
 &  & \multirow{2}{*}{10} & 25 & 251 & 9.65E-09 & 6.17E-09 & 2.63E+03 \\ \cline{4-8}
 &  &  & 50 & 37 & 8.47E-09 & 4.64E-09 & 4.49E+03 \\ \cline{3-8}
 &  & \multirow{2}{*}{-10} & 25 & 134 & 9.99E-09 & 3.20E-09 & 4.58E+03 \\ \cline{4-8}
 &  &  & 50 & 21 & 9.53E-09 & 2.44E-09 & 4.64E+03 \\ \hline
\multirow{8}{*}{lp22} & \multirow{8}{*}{\begin{tabular}[c]{@{}c@{}}2958\\ $\times$\\ 16392\end{tabular}} & \multirow{2}{*}{5} & 25 & 24 & 9.20E-09 & 1.03E-09 & 4.55E+01 \\ \cline{4-8}
 &  &  & 50 & 8 & 8.03E-09 & 9.55E-10 & 4.60E+01 \\ \cline{3-8}
 &  & \multirow{2}{*}{-5} & 25 & 14 & 9.67E-09 & 1.03E-09 & 3.72E+01 \\ \cline{4-8}
 &  &  & 50 & 5 & 7.92E-09 & 8.42E-10 & 4.62E+01 \\ \cline{3-8}
 &  & \multirow{2}{*}{10} & 25 & 38 & 9.43E-09 & 1.02E-09 & 4.42E+01 \\ \cline{4-8}
 &  &  & 50 & 10 & 9.74E-09 & 1.03E-09 & 4.50E+01 \\ \cline{3-8}
 &  & \multirow{2}{*}{-10} & 25 & 24 & 9.54E-09 & 1.24E-09 & 3.21E+01 \\ \cline{4-8}
 &  &  & 50 & 7 & 8.76E-09 & 1.06E-09 & 4.62E+01 \\ \hline
\multirow{8}{*}{lp\_maros\_r7} & \multirow{8}{*}{\begin{tabular}[c]{@{}c@{}}3136\\ $\times$\\ 9408\end{tabular}} & \multirow{2}{*}{5} & 25 & 20 & 9.41E-09 & 4.88E-09 & 5.06E+00 \\ \cline{4-8}
 &  &  & 50 & 7 & 9.88E-09 & 1.34E-08 & 5.07E+00 \\ \cline{3-8}
 &  & \multirow{2}{*}{-5} & 25 & 55 & 9.58E-09 & 5.00E-09 & 5.03E+00 \\ \cline{4-8}
 &  &  & 50 & 13 & 8.21E-09 & 5.74E-09 & 5.08E+00 \\ \cline{3-8}
 &  & \multirow{2}{*}{10} & 25 & 73 & 9.31E-09 & 1.28E-08 & 4.63E+00 \\ \cline{4-8}
 &  &  & 50 & 17 & 9.49E-09 & 1.29E-08 & 5.00E+00 \\ \cline{3-8}
 &  & \multirow{2}{*}{-10} & 25 & 84 & 9.95E-09 & 5.21E-09 & 5.04E+00 \\ \cline{4-8}
 &  &  & 50 & 18 & 9.54E-09 & 6.49E-09 & 5.06E+00 \\ \hline
\multirow{8}{*}{model8} & \multirow{8}{*}{\begin{tabular}[c]{@{}c@{}}2896\\ $\times$\\ 6464\end{tabular}} & \multirow{2}{*}{5} & 25 & 14 & 7.67E-09 & 2.64E-09 & 2.94E+01 \\ \cline{4-8}
 &  &  & 50 & 5 & 5.66E-09 & 2.89E-09 & 1.16E+02 \\ \cline{3-8}
 &  & \multirow{2}{*}{-5} & 25 & 46 & 9.81E-09 & 3.14E-09 & 1.23E+02 \\ \cline{4-8}
 &  &  & 50 & 13 & 8.87E-09 & 2.86E-09 & 1.24E+02 \\ \cline{3-8}
 &  & \multirow{2}{*}{10} & 25 & 27 & 9.29E-09 & 3.18E-09 & 9.51E+01 \\ \cline{4-8}
 &  &  & 50 & 8 & 7.61E-09 & 3.93E-09 & 7.90E+01 \\ \cline{3-8}
 &  & \multirow{2}{*}{-10} & 25 & 86 & 9.86E-09 & 3.37E-09 & 1.22E+02 \\ \cline{4-8}
 &  &  & 50 & 18 & 9.84E-09 & 3.33E-09 & 1.19E+02 \\ \hline
\multirow{8}{*}{rosen10} & \multirow{8}{*}{\begin{tabular}[c]{@{}c@{}}2056\\ $\times$\\ 6152\end{tabular}} & \multirow{2}{*}{5} & 25 & 55 & 9.39E-09 & 9.99E-11 & 3.04E+02 \\ \cline{4-8}
 &  &  & 50 & 19 & 9.58E-09 & 9.67E-11 & 3.17E+02 \\ \cline{3-8}
 &  & \multirow{2}{*}{-5} & 25 & 5 & 9.92E-09 & 2.13E-09 & 3.65E+01 \\ \cline{4-8}
 &  &  & 50 & 2 & 4.65E-09 & 7.73E-10 & 5.66E+01 \\ \cline{3-8}
 &  & \multirow{2}{*}{10} & 25 & 352 & 9.64E-09 & 6.01E-10 & 3.27E+02 \\ \cline{4-8}
 &  &  & 50 & 49 & 8.88E-09 & 1.13E-10 & 3.27E+02 \\ \cline{3-8}
 &  & \multirow{2}{*}{-10} & 25 & 13 & 9.06E-09 & 4.65E-10 & 1.37E+01 \\ \cline{4-8}
 &  &  & 50 & 4 & 6.21E-09 & 1.04E-09 & 3.83E+01 \\ \hline
\end{tabular}
}
\caption{Experimental results of IRJBD}
\label{tab:results}
\end{table}

For all the problems but numerically rank deficient kneser\_8\_3\_1$^{\mathrm T}$,
the IRJBD algorithm successfully computed
the $l$ desired GSVD components. It is seen from Table \ref{tab:results} that $Res$ and $Res\_b$ are comparable in size and $Res<Res\_b$ often, showing that bound (\ref{bound of residual}) is tight and (\ref{stopping criterion}) is a reliable stopping criterion. These cases precisely correspond to modestly sized $\|\underline{B}_k^{-1}\|$ and $\|\widehat B_k^{-1}\|$. On the other hand, if the product $\|\underline{B}_k^{-1}\|\|\widehat B_k^{-1}\|$ is large, as indicated by kneser\_8\_3\_1$^{\mathrm T}$, the stopping criterion (\ref{stopping criterion}) is no longer reliable. For $A=\text{kneser\_8\_3\_1}^{\mathrm T}$, $\{A,L\}$ has at least one trivial zero GSVD component. As
a consequence, when the smallest GSVD components are of interest, the smallest eigenvalue of $\underline{B}_k^{\mathrm T}\underline{B}_k$ becomes very small as $k$ increases because the smallest eigenvalue of $Q_AQ_A^{\mathrm T}$ is tiny, which causes $\|\underline{B}_k^{-1}\|$ uncontrollably 
large, so is $\|\underline{B}_k^{-1}\|\|\widehat B_k^{-1}\|$. Thus, the actual residual norms may not be small, which confirms Theorem \ref{Th:bound of residual in finite}, and 
the results are in accordance with our analysis on the reliability issue of $Res_b$ after \Cref{Def:relative residual}. 

The above experiments and analysis naturally pose a crucial question: How to make the JBD method unaffected by zero and infinite GSVD components and only compute desired nontrivial ones.

Taking the computation of ten largest and smallest GSVD components of two matrix pairs as an example, we depict Figure \ref{fig:flower_5_4andlp22} and show how the upper bounds of the relative residual norms change as restarts proceed. From the figure, we observe irregular convergence phenomena, though IRJBD overall converged quite fast.

\begin{figure}
\centering
\includegraphics[width=\linewidth]{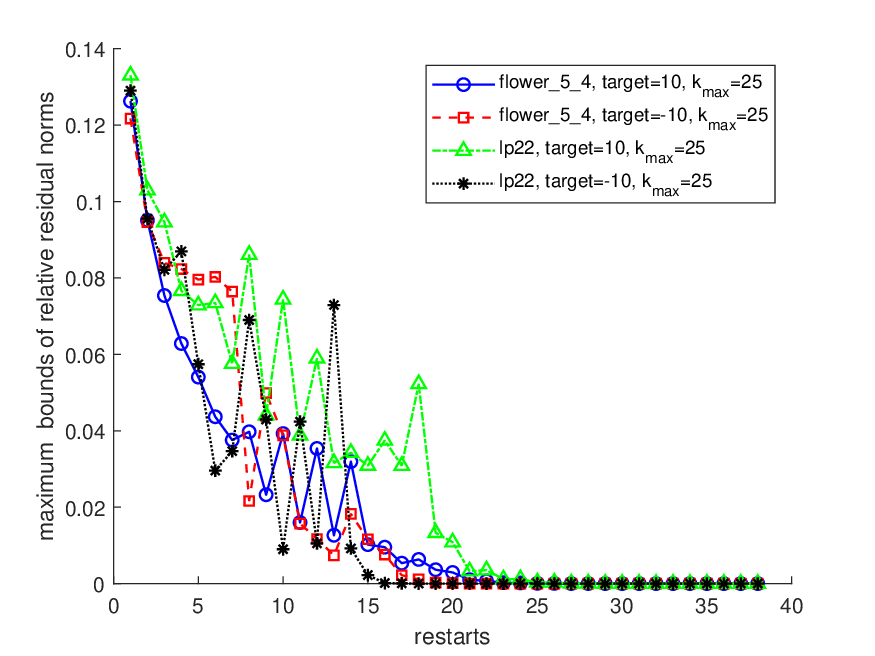}
\caption{Results of computing the ten largest and smallest GSVD components of two matrix pairs}
\label{fig:flower_5_4andlp22}
\end{figure}

\begin{table}[]
\centering
\resizebox{\textwidth}{!}{
\begin{tabular}{|l|c|l|l|l|l|l|l|l|}
\hline
$A$ & $m\times n$ & $target$ & $k_{\max}$ & $iter_I$ & $iter_T$ & $Res_I$ & $Res_T$ & $speedup$(\%) \\ \hline
\multirow{8}{*}{flower\_5\_4} & \multirow{8}{*}{\begin{tabular}[c]{@{}c@{}}5226\\ $\times$\\ 14721\end{tabular}} & \multirow{2}{*}{5} & 25 & 17 & 21 & 4.55E-09 & 1.83E-09 & 19.0476 \\ \cline{4-9}
 &  &  & 50 & 6 & 9 & 5.40E-09 & 1.19E-09 & 33.3333 \\ \cline{3-9}
 &  & \multirow{2}{*}{-5} & 25 & 14 & 15 & 4.49E-09 & 1.22E-09 & 6.6667 \\ \cline{4-9}
 &  &  & 50 & 5 & 7 & 5.17E-09 & 3.70E-10 & 28.5714 \\ \cline{3-9}
 &  & \multirow{2}{*}{10} & 25 & 38 & 37 & 4.33E-09 & 2.25E-09 & -2.7027 \\ \cline{4-9}
 &  &  & 50 & 10 & 15 & 5.16E-09 & 1.73E-09 & 33.3333 \\ \cline{3-9}
 &  & \multirow{2}{*}{-10} & 25 & 28 & 23 & 5.15E-09 & 1.55E-09 & -21.7391 \\ \cline{4-9}
 &  &  & 50 & 8 & 10 & 4.48E-09 & 7.88E-10 & 20.0000 \\ \hline
\multirow{8}{*}{kneser\_8\_3\_1} & \multirow{8}{*}{\begin{tabular}[c]{@{}c@{}}15681\\ $\times$\\ 15737\end{tabular}} & \multirow{2}{*}{5} & 25 & 10 & 13 & 4.48E-09 & 2.18E-09 & 23.0769 \\ \cline{4-9}
 &  &  & 50 & 4 & 6 & 4.35E-09 & 1.26E-09 & 33.3333 \\ \cline{3-9}
 &  & \multirow{2}{*}{-5} & 25 & 113 & 122 & 0.0100 & 0.0313 & 7.3770 \\ \cline{4-9}
 &  &  & 50 & 33 & 127 & 0.0062 & 0.0584 & 74.0157 \\ \cline{3-9}
 &  & \multirow{2}{*}{10} & 25 & 67 & 61 & 3.74E-07 & 3.24E-09 & -9.8361 \\ \cline{4-9}
 &  &  & 50 & 14 & 19 & 1.37E-06 & 2.56E-09 & 26.3158 \\ \cline{3-9}
 &  & \multirow{2}{*}{-10} & 25 & 284 & 339 & 0.0228 & 0.0476 & 16.2242 \\ \cline{4-9}
 &  &  & 50 & 61 & 314 & 0.0075 & 0.1944 & 80.5732 \\ \hline
\multirow{8}{*}{l30} & \multirow{8}{*}{\begin{tabular}[c]{@{}c@{}}2701\\ $\times$\\ 16281\end{tabular}} & \multirow{2}{*}{5} & 25 & 264 & 258 & 6.21E-09 & 1.95E-09 & -2.3256 \\ \cline{4-9}
 &  &  & 50 & 48 & 50 & 6.23E-09 & 1.67E-09 & 4.0000 \\ \cline{3-9}
 &  & \multirow{2}{*}{-5} & 25 & 63 & 80 & 2.62E-09 & 2.00E-09 & 21.2500 \\ \cline{4-9}
 &  &  & 50 & 14 & 25 & 3.19E-09 & 1.48E-09 & 44.0000 \\ \cline{3-9}
 &  & \multirow{2}{*}{10} & 25 & 251 & 259 & 6.17E-09 & 1.95E-09 & 3.0888 \\ \cline{4-9}
 &  &  & 50 & 37 & 50 & 4.64E-09 & 1.38E-09 & 26.0000 \\ \cline{3-9}
 &  & \multirow{2}{*}{-10} & 25 & 134 & 129 & 3.20E-09 & 2.00E-09 & -3.8760 \\ \cline{4-9}
 &  &  & 50 & 21 & 34 & 2.44E-09 & 1.48E-09 & 38.2353 \\ \hline
\multirow{8}{*}{lp22} & \multirow{8}{*}{\begin{tabular}[c]{@{}c@{}}2958\\ $\times$\\ 16392\end{tabular}} & \multirow{2}{*}{5} & 25 & 24 & 27 & 1.03E-09 & 5.36E-10 & 11.1111 \\ \cline{4-9}
 &  &  & 50 & 8 & 12 & 9.55E-10 & 2.32E-10 & 33.3333 \\ \cline{3-9}
 &  & \multirow{2}{*}{-5} & 25 & 14 & 16 & 1.03E-09 & 6.70E-10 & 12.5000 \\ \cline{4-9}
 &  &  & 50 & 5 & 8 & 8.42E-10 & 4.76E-10 & 37.5000 \\ \cline{3-9}
 &  & \multirow{2}{*}{10} & 25 & 38 & 35 & 1.02E-09 & 6.83E-10 & -8.5714 \\ \cline{4-9}
 &  &  & 50 & 10 & 15 & 1.03E-09 & 5.39E-10 & 33.3333 \\ \cline{3-9}
 &  & \multirow{2}{*}{-10} & 25 & 24 & 23 & 1.24E-09 & 9.23E-10 & -4.3478 \\ \cline{4-9}
 &  &  & 50 & 7 & 10 & 1.06E-09 & 4.76E-10 & 30.0000 \\ \hline
\multirow{8}{*}{lp\_maros\_r7} & \multirow{8}{*}
{\begin{tabular}[c]{@{}c@{}}3136\\ $\times$\\ 9408\end{tabular}} & \multirow{2}{*}{5} & 25 & 20 & 26 & 4.88E-09 & 6.62E-09 & 23.0769 \\ \cline{4-9}
 &  &  & 50 & 7 & 8 & 1.34E-08 & 0.0658 & 12.5000 \\ \cline{3-9}
 &  & \multirow{2}{*}{-5} & 25 & 55 & 45 & 5.00E-09 & 6.12E-09 & -22.2222 \\ \cline{4-9}
 &  &  & 50 & 13 & 19 & 5.74E-09 & 4.93E-09 & 31.5789 \\ \cline{3-9}
 &  & \multirow{2}{*}{10} & 25 & 73 & 78 & 1.28E-08 & 6.62E-09 & 6.4103 \\ \cline{4-9}
 &  &  & 50 & 17 & 25 & 1.29E-08 & 0.0658 & 32.0000 \\ \cline{3-9}
 &  & \multirow{2}{*}{-10} & 25 & 84 & 57 & 5.21E-09 & 6.12E-09 & -47.3684 \\ \cline{4-9}
 &  &  & 50 & 18 & 23 & 6.49E-09 & 5.67E-09 & 21.7391 \\ \hline
\multirow{8}{*}{model8} & \multirow{8}{*}{\begin{tabular}[c]{@{}c@{}}2896\\ $\times$\\ 6464\end{tabular}} & \multirow{2}{*}{5} & 25 & 14 & 18 & 2.64E-09 & 1.15E-09 & 22.2222 \\ \cline{4-9}
 &  &  & 50 & 5 & 8 & 2.89E-09 & 1.70E-10 & 37.5000 \\ \cline{3-9}
 &  & \multirow{2}{*}{-5} & 25 & 46 & 42 & 3.14E-09 & 1.41E-09 & -9.5238 \\ \cline{4-9}
 &  &  & 50 & 13 & 18 & 2.86E-09 & 1.12E-09 & 27.7778 \\ \cline{3-9}
 &  & \multirow{2}{*}{10} & 25 & 27 & 26 & 3.18E-09 & 1.31E-09 & -3.8462 \\ \cline{4-9}
 &  &  & 50 & 8 & 11 & 3.93E-09 & 4.53E-10 & 27.2727 \\ \cline{3-9}
 &  & \multirow{2}{*}{-10} & 25 & 86 & 61 & 3.37E-09 & 1.41E-09 & -40.9836 \\ \cline{4-9}
 &  &  & 50 & 18 & 21 & 3.33E-09 & 8.84E-10 & 14.2857 \\ \hline
\multirow{8}{*}{rosen10} & \multirow{8}{*}{\begin{tabular}[c]{@{}c@{}}2056\\ $\times$\\ 6152\end{tabular}} & \multirow{2}{*}{5} & 25 & 55 & 97 & 9.99E-11 & 7.25E-12 & 43.2990 \\ \cline{4-9}
 &  &  & 50 & 19 & 38 & 9.67E-11 & 6.42E-06 & 50.0000 \\ \cline{3-9}
 &  & \multirow{2}{*}{-5} & 25 & 5 & 11 & 2.13E-09 & 2.04E-11 & 54.5455 \\ \cline{4-9}
 &  &  & 50 & 2 & 3 & 7.73E-10 & 8.23E-13 & 33.3333 \\ \cline{3-9}
 &  & \multirow{2}{*}{10} & 25 & 352 & 328 & 6.01E-10 & 7.25E-12 & -7.3171 \\ \cline{4-9}
 &  &  & 50 & 49 & 60 & 1.13E-10 & 1.54E-05 & 18.3333 \\ \cline{3-9}
 &  & \multirow{2}{*}{-10} & 25 & 13 & 13 & 4.65E-10 & 3.73E-11 & 0.0000 \\ \cline{4-9}
 &  &  & 50 & 4 & 5 & 1.04E-09 & 1.30E-11 & 20.0000 \\ \hline
\end{tabular}
}
\caption{An efficiency comparison of IRJBD and TRJBD}
\label{tab:results of TRJBD}
\end{table}

We also test the code of TRJBD in SLEPc \cite{TRJBD}, the Scalable Library for Eigenvalue Problem
\cite{SLEPc} on the above problems. To be fair, we took the same
stopping tolerance, maximum restarts and
starting vectors for TRJBD as those for IRJBD. Since TRJBD
was written in the more efficient C language and
was optimized in some sense while our IRJBD is in MATLAB
language, comparing runtime is not a right way and is generally misleading. However, since the
computational cost of one restart of IRJBD is slightly lower than that of TRJBD,
for the same starting vectors and subspace
dimension, the overall efficiency can be measured in terms of the total iter's, which are independent of
programming languages. Table \ref{tab:results of
TRJBD} lists the results obtained by IRJBD and TRJBD, where the subscripts I and T stands for IRJBD and
TRJBD, respectively, $iter_I$ and $iter_T$ are the numbers of restarts used by IRJBD and TRJBD when they
terminated. The residual used for TRJBD is defined by
$$
\|r_i^{GSVD}\|=\sqrt{\|s_i^{(k)}A^{\mathrm T}y_i^{(k)}-c_i^{(k)}L^{\mathrm T}Lx_i^{(k)}\|^2+\|c_i^{(k)}L^{\mathrm T}z_i^{(k)}-s_i^{(k)}A^{\mathrm T}Ax_i^{(k)}\|^2}
$$
and it is proved in \cite{TRJBD} that $\|r_i^{GSVD}\|=\|r_{i,3}^{(k)}\|$, where $r_{i,3}^{(k)}$ is defined as in (\ref{residual}). Since $\|r_i^{(k)}\|=\|r_{i,3}^{(k)}\|$, we have $\|r_i^{GSVD}\|=\|r_i^{(k)}\|$.
Therefore,
$$
speedup=(iter_T-iter_I)/iter_T
$$
indicates how much faster IRJBD is than TRJBD.

As is seen from Table \ref{tab:results of TRJBD}, in most cases, our IRJBD algorithm used fewer and often
much fewer restarts to converge than the TRJBD algorithm did, as the {\em speedup}s indicate; e.g, 54.5\% fewer for rosen10. For the
computation of the smallest GSVD components
of the numerical rank-deficient matrix $A=\text{kneser\_8\_3\_1}$, neither IRJBD nor TRJBD worked, as is
expected. We find that for the well-conditioned $A=\text{rosen10}$, $target=5,\,10$ and $k_{\max}=50$, TRJBD stopped within restarts $maxit$ but the actual residual norms are much larger than $tol$.
There should be some bugs in the code of TRJBD. Overall, IRJBD often performs considerably better than TRJBD.

\section{Conclusion}

We have extended the implicit restarting technique to the JBD process and developed an IRJBD algorithm to 
compute several largest or smallest GSVD components of a large regular matrix pair. We have proved that 
the JBD method cannot compute zero and infinite GSVD components, which is unlike the JD GSVDsolver 
\cite{Huang-Numerical-experiments}. We have introduced a general residual, 
established compact bounds for it in both exact and finite precision arithmetic, and 
revealed the reliability requirement of using it in finite precision. The bound is 
important in computation since it is efficient to use and avoids the expensive computation 
of each approximate right generalized singular vector before convergence.
Numerical experiments have confirmed our theoretical results and analysis, and shown that our 
IRJBD algorithm works well if there is no zero or infinite GSVD components. We have also 
illustrated that IRJBD is more efficient and often performs much better than TRJBD.

There are some important problems to be solved. As we have proved, addressed and numerically justified,
the JBD method and the resulting IRJBD and TRJBD algorithms cannot compute zero and infinite GSVD 
components. Because of the appearance of zero or infinite GSVD components,
the JBD method encountered a severe difficulty when computing other nontrivial finite largest or smallest GSVD components. How to remove the fatal effects of zero or infinite generalized singular value on
the JBD method is extremely significant and important. More generally, zero and infinite
GSVD components have similar effects on other solvers such as JD GSVDsolvers.
Certainly, this issue is
challenging and will constitute our future work.
We have also observed irregular convergence behaviors of the JBD method,
which should be due to intrinsic possible non-convergence of the standard Rayleigh--Ritz
method \cite{Jia-poly-refined,Jia-IRRBL}. As has been done for the large eigenvalue problems and SVD problems, we will pursue
a refined extraction-based JBD method and corresponding implicitly restarted algorithms in the near future.

\section*{Declarations}

The two authors declare that they have no conflict of interest, and they read and approved the final manuscript.

\end{document}